\documentclass[11pt,reqno]{amsart}

\usepackage{amsmath,amssymb,mathtools}
\usepackage{booktabs}
\usepackage{enumitem}
\usepackage{hyperref}
\usepackage{physics}
\usepackage{tikz}
\usetikzlibrary{arrows.meta,calc,decorations.pathreplacing,positioning}
\usepackage{aliascnt}
\usepackage[capitalize,nameinlink]{cleveref}

\hypersetup{
  colorlinks=true,
  linkcolor=blue,
  citecolor=blue,
  urlcolor=blue
}

\newtheorem{theorem}{Theorem}[section]

\newaliascnt{lemma}{theorem}
\newtheorem{lemma}[lemma]{Lemma}
\aliascntresetthe{lemma}
\crefname{lemma}{Lemma}{Lemmas}
\Crefname{lemma}{Lemma}{Lemmas}

\newaliascnt{proposition}{theorem}
\newtheorem{proposition}[proposition]{Proposition}
\aliascntresetthe{proposition}
\crefname{proposition}{Proposition}{Propositions}
\Crefname{proposition}{Proposition}{Propositions}

\newaliascnt{corollary}{theorem}

\aliascntresetthe{corollary}
\crefname{corollary}{Corollary}{Corollaries}
\Crefname{corollary}{Corollary}{Corollaries}

\newaliascnt{claim}{theorem}
\newtheorem{claim}[claim]{Claim}
\aliascntresetthe{claim}
\crefname{claim}{Claim}{Claims}
\Crefname{claim}{Claim}{Claims}

\theoremstyle{definition}
\newaliascnt{definition}{theorem}
\newtheorem{definition}[definition]{Definition}
\aliascntresetthe{definition}
\crefname{definition}{Definition}{Definitions}
\Crefname{definition}{Definition}{Definitions}

\newtheorem{remark}[definition]{Remark}

\newcommand{\R}{\mathbb R}
\newcommand{\N}{\mathbb N}
\newcommand{\eps}{\varepsilon}

\title[Dimension-free covering functional estimates]
{Dimension-free estimates for covering functionals of simplices and \(\ell_p\) balls}

\author{Feifei Chen}
\address{School of Mathematics, North University of China, Taiyuan 030051,
  China}
\email{chernfeifei@163.com}

\author{Chan He}
\address{School of Mathematics, North University of China, Taiyuan 030051,
  China}
\email{hechan@nuc.edu.cn}

\author{Senlin Wu}
\address{School of Mathematics, North University of China, Taiyuan 030051,
  China}
\email{wusenlin@nuc.edu.cn}



\date{\today}

\begin{document}

\begin{abstract}
  We study \(\Gamma_{2^n}(K)\), the least positive number \(\gamma>0\) such that
  an \(n\)-dimensional convex body \(K\) can be covered by \(2^n\) translates of
  \(\gamma K\). The inequality \(\Gamma_{2^n}(K)<1\) holds if and only if \(K\)
  can be covered by \(2^n\) translates of its interior. Thus upper bounds on
  \(\Gamma_{2^n}(K)\) provide quantitative information on Hadwiger's covering
  conjecture. It is known that \(\Gamma_{2^n}(K)\geq1/2\) for every \(K\), with
  equality precisely for bodies affinely equivalent to \([-1,1]^n\). For
  \(n\)-simplices \(\Delta_n\), we prove that, surprisingly,
  \(\lim_{n\to\infty}\Gamma_{2^n}(\Delta_n)=1/2\). For the cross-polytope
  \(B_1^n\), we show that \(\Gamma_{2^n}(B_1^n)\leq5/6\) holds for all
  \(n\geq2\), and that
  \(\limsup_{n\to\infty}\Gamma_{2^n}(B_1^n)\leq0.641\cdots\). Finally, we prove
  a uniform quantitative Hadwiger-type estimate for \(\ell_p\) balls: there is a
  constant \(\kappa_*<1\) such that \(\Gamma_{2^n}(B_p^n)\leq\kappa_*\) for all
  \(n\geq2\) and all \(p\in[1,\infty]\).
\end{abstract}

\maketitle

\section{Introduction}

Denote by \(\mathcal{K}^n\) the set of all \emph{convex bodies} (compact and
convex sets having interior points) in \(\mathbb{R}^n\). For each \(K\in
\mathcal{K}^n\), let \(c(K)\) be the least number of translates of its interior
that can cover \(K\). Equivalently, \(c(K)\) is the least number of smaller
homothetic copies of \(K\) needed to cover \(K\). Hadwiger's covering conjecture
asserts that \(c(K)\leq 2^n,~\forall K\in \mathcal{K}^n\) and \(c(K)=2^n\) if
and only if \(K\) is affinely equivalent to \([0,1]^n\), cf.
\cite{Hadwiger1957}. The problem was solved in the planar case by Levi
\cite{Levi1955}, but remains open when \(n\geq3\); see
\cite{BoltyanskiMartiniSoltan1997,BrassMoserPach2005,BezdekKhan2018}. Recent
work on general upper bounds for \(c(K)\) can be found in
\cite{HuangSlomkaTkoczVritsiou2022} and \cite{CamposHintumMorrisTiba2024}. For
related discussions of covering, illumination, and separation viewpoints, see
\cite{Bezdek1991AffineSubspaces,Bezdek1992Relatives,Bezdek1993HadwigerLevi,
  Bezdek2006Illumination,BezdekKhan2018}.

For each integer \(m\geq1\) and each \(K\in\mathcal{K}^n\), set \([m]=[1,m]\cap
\mathbb{Z}\) and
\begin{displaymath}
  \Gamma_m(K)=\inf\left\{\gamma>0:\ \exists c_1,\ldots,c_m\in\R^n\ \text{such that}\ K\subseteq \bigcup_{i\in[m]}(c_i+\gamma K)\right\}.
\end{displaymath}
It can be verified that
\begin{displaymath}
  c(K)\leq2^n \quad\Longleftrightarrow\quad \Gamma_{2^n}(K)<1.
\end{displaymath}
This equivalence underlies Zong's program for Hadwiger's covering conjecture;
see \cite{Zong2010}.

We call the map \(\Gamma_{2^n}:~\mathcal{K}^n\to [0,1],~K\mapsto \Gamma_{2^n}(K)\) the
\emph{covering functional}. By \cite[Theorem 1.7]{HeMartiniWu2016},
\(\Gamma_{2^n}(K)\geq 1/2,~\forall K\in \mathcal{K}^n\), and
\(\Gamma_{2^n}(K)=1/2\) if and only if \(K\) is affinely equivalent to
\([-1,1]^n\). The covering-functional viewpoint was developed in
\cite{HeMartiniWu2016,WuHe2019}; related quantitative, computational, and
polytopal work can be found in \cite{Zong2010,RogersZong1997,ArtsteinAvidanSlomka2015,HeLvMartiniWu2023,LiMengWu2022,
  YuGaoHeWu2023,ChenGaoLiWu2025,LianXiaXueZhang2026}.
We study the behavior of the covering functional for
\begin{displaymath}
  \Delta_n=\left\{x\in\R^n:\ x_i\geq0,\ \sum_{i=1}^n
    x_i\leq1\right\}\quad\text{and}\quad B_p^n=\{x\in\R^n:\ \norm{x}_p\leq1\},
\end{displaymath}
where \(\norm{\cdot}_p\) is the standard \(p\)-norm on \(\mathbb{R}^n\).

For the two families studied here, Lian, Xia, Xue and Zhang obtained in
\cite{LianXiaXueZhang2026} that
\begin{displaymath}
  \limsup_{n\to\infty}\Gamma_{2^n}(\Delta_n)\leq 0.773\cdots\quad \text{and}\quad \limsup_{n\to\infty}\Gamma_{2^n}(B_1^n)\leq 0.824\cdots,
\end{displaymath}
see also \cite[Corollary 2.8]{ChenGaoLiWu2025}. Li, Meng and Wu showed
in \cite[Proposition 5]{LiMengWu2022} that
\begin{displaymath}
  \limsup_{n\to\infty}\Gamma_{2^n}(\Delta_n)\leq 0.772\cdots .
\end{displaymath}
As mentioned above, for an \(n\)-dimensional convex body \(K\),
\(\Gamma_{2^n}(K)=1/2\) if and only if \(K\) is affinely equivalent to
\(B_\infty^n=[-1,1]^n\). Our first main result shows that, surprisingly, the
same asymptotic value holds for \(\Delta_n\).

\begin{theorem}
  \label{thm:simplex-half}
  We have
  \begin{displaymath}
    \lim_{n\to\infty}\Gamma_{2^n}(\Delta_n)=\frac{1}{2}.
  \end{displaymath}
\end{theorem}

The second result provides a universal upper bound for \(\Gamma_{2^n}(B_1^n)\).

\begin{theorem}
  \label{thm:b1-main}
  For every \(n\geq2\),
  \begin{displaymath}
    \Gamma_{2^n}(B_1^n)\leq \frac{5}{6}.
  \end{displaymath}
\end{theorem}

Let \(\gamma_{\rm sgn}\in(1/2,1)\) be the unique solution of
\begin{displaymath}
  -\ln\gamma+(1-\gamma)\ln2=\ln2.
\end{displaymath}
Numerically,
\begin{displaymath}
  \gamma_{\rm sgn}=0.641\cdots .
\end{displaymath}

\begin{theorem}
  \label{thm:b1-asymptotic}
  We have
  \begin{displaymath}
    \limsup_{n\to\infty}\Gamma_{2^n}(B_1^n)\leq\gamma_{\rm sgn}.
  \end{displaymath}
\end{theorem}

\begin{proposition}
  \label{prop:fixed-p}
  For every \(p\in[1,\infty)\),
  \begin{displaymath}
    \sup_{n\geq2}\Gamma_{2^n}(B_p^n)\leq \left(\frac56\right)^{1/p}.
  \end{displaymath}
\end{proposition}

The upper bound in \Cref{prop:fixed-p} is not uniform in \(p\). The last main
result is the following uniform estimate.

\begin{theorem}
  \label{thm:uniform-p}
  There exists a constant \(\kappa_*<1\) such that
  \begin{displaymath}
    \Gamma_{2^n}(B_p^n)\leq \kappa_*,~\forall n\geq 2,~\forall p\in[1,\infty].
  \end{displaymath}
\end{theorem}

\medskip

In what follows, we shall use the fact that \(\Gamma_m\) is affine invariant and
monotone in \(m\): if \(m_1\leq m_2\), then \(\Gamma_{m_2}(K)\leq
\Gamma_{m_1}(K)\).

For \(1\leq p<\infty\),
\begin{equation}
  \label{eq:cube-cover}
  \Gamma_{2^n}(B_p^n)\leq \frac{n^{1/p}}2 .
\end{equation}
Indeed, with \(c_\varepsilon=\varepsilon/2\) and \(\varepsilon\in\{-1,1\}^n\),
we have
\begin{displaymath}
  B_p^n\subseteq B_\infty^n=\bigcup_{\varepsilon\in\{-1,1\}^n}\left(c_\varepsilon+\frac12 B_\infty^n\right)\subseteq\bigcup_{\varepsilon\in\{-1,1\}^n}\left(c_\varepsilon+\frac{n^{1/p}}2 B_p^n\right).
\end{displaymath}

In the sequel, the dimension \(n\) is always assumed to be at least \(2\).

\section{Asymptotic estimates for \(\Delta_n\) and \(B_1^n\)}
\label{sec:asymptotic-coordinate-coverings}

The notation and parameters introduced in this section are used only in the two
asymptotic covering arguments. The choice \(\varrho=1/2\) will be used for
\(\Delta_n\), while the choice \(\varrho=\gamma_{\rm sgn}\) will be used for
\(B_1^n\).

The proofs of \Cref{thm:simplex-half} and \Cref{thm:b1-asymptotic} use a common
coordinate discretization and a probabilistic construction based on random
injections into \([n]\). The decompositions \eqref{eq:common-simplex-decomp} and
\eqref{eq:common-cross-decomp} reduce the original problems to the problems of
covering \(P(n,\gamma)\) and \(Q(n,\gamma)\); the remaining portions of
\(\Delta_n\) and \(B_1^n\) are covered by the translates of \(\gamma \Delta_n\)
and \(\gamma B_1^n\) with translation vectors in
\[
  \{0\}\cup\{(1-\gamma)e_i:i\in[n]\}
  \quad\text{and}\quad
  \{0\}\cup\{\pm(1-\gamma)e_i:i\in[n]\},
\]
respectively. The layers in \eqref{eq:common-layers} are chosen so that
\Cref{lem:coordinate-replacement}, applied to \(x\) for \(\Delta_n\) or to
\((|x_1|,\ldots,|x_n|)\) for \(B_1^n\), produces a point \(q\) satisfying
\(d_j=t_j-\gamma<\lambda\norm{q}_1\) by
\eqref{eq:common-layer-ineq}. Hence the construction in
\Cref{subsec:finite-coordinate-values} applies with \(d=d_j\) and, for each
fixed value of \(\mathrm{mult}(q)\), produces a random vector \(u\) with
\(\norm{u}_1=d_j\). The event \(E(q)\) in
\Cref{lem:coordinate-assignment-probability} has probability \(P(q)\) given by
\eqref{eq:Pq} and implies \(u\leq q\); \Cref{lem:integral-exponent} gives a
lower bound for \(P(q)\) that is uniform in \(q\).

For \(\Delta_n\), \Cref{lem:multiplicity-count},
\eqref{eq:probability-estimate}, and \eqref{eq:union-probability} yield a
deterministic family of fewer than \(2^n-(n+1)\) translation vectors. The
relations \(u\leq q\leq x\), \(\norm{u}_1=d_j=t_j-\gamma\), and \(\norm{x}_1\leq
t_j\) then allow the application of \Cref{lem:simplex-containment}. For
\(B_1^n\), independent signs are assigned to the nonzero coordinates of \(u\),
and \Cref{lem:crosspolytope-probability}, \Cref{lem:multiplicity-count}, and
\eqref{eq:probability-estimate} yield a deterministic family of fewer than
\(2^n-(2n+1)\) translation vectors. The resulting vector \(u\) satisfies
\(|u_i|\leq|x_i|\) and \(\operatorname{sgn}(u_i)=\operatorname{sgn}(x_i)\)
whenever \(u_i\neq0\), so \(\norm{x-u}_1=\norm{x}_1-d_j\leq\gamma\). Combining
the two constructions with \eqref{eq:common-simplex-decomp} and
\eqref{eq:common-cross-decomp} completes the proofs.

\subsection{Parameters and decompositions of \texorpdfstring{\(\Delta_n\)}{Delta
    n} and \texorpdfstring{\(B_1^n\)}{B1n}}
\label{subsec:common-constants}

Fix
\begin{displaymath}
  \varrho\in\left\{\frac12,\gamma_{\rm sgn}\right\}, \quad
  \gamma\in(\varrho,1),\quad\text{and}\quad \lambda\in(1-\gamma,1-\varrho)\subset\left(0,\frac12\right).
\end{displaymath}
Let \(\eta\) be a number in \((0,1)\) such that
\begin{displaymath}
  \alpha:=\lambda(1-\eta)>1-\gamma. 
\end{displaymath}
Pick
\begin{displaymath}
  \alpha'\in (1-\gamma,\alpha).
\end{displaymath}
Set \(u_0=\gamma\) and
\begin{displaymath}
  u_{j+1}=\gamma+\alpha' u_j,~\forall j\in[0,\infty)\cap \mathbb{Z}.
\end{displaymath}
Equivalently,
\begin{displaymath}
  u_j=\gamma\frac{1-(\alpha')^{j+1}}{1-\alpha'},~\forall j\in[0,\infty)\cap \mathbb{Z}.
\end{displaymath}
Since \(\alpha'>1-\gamma\), \(\{u_j\}_{j=0}^\infty\) increases to \(\gamma/(1-\alpha')>1\). Hence the integer
\begin{displaymath}
  N=\min\{j\in [0,\infty)\cap \mathbb{Z}:\ u_j\geq1\}
\end{displaymath}
is well defined. Put
\begin{displaymath}
  t_j=u_j,~\forall j\in[0,N)\cap \mathbb{Z} \quad \text{and}\quad t_N=1.
\end{displaymath}
Then
\begin{equation}
  \label{eq:common-layers}
  \gamma=t_0<t_1<\cdots<t_N=1.
\end{equation}
Set
\begin{displaymath}
  d_j=t_j-\gamma,~\forall j\in[N].
\end{displaymath}
Then
\begin{gather*}
  d_j=t_j-\gamma=\alpha't_{j-1}<\alpha t_{j-1},~\forall j\in[N-1],\\
  d_N=1-\gamma\leq\alpha'u_{N-1}<\alpha t_{N-1}.
\end{gather*}
Hence
\begin{equation}
  \label{eq:common-layer-ineq}
  d_j<\alpha t_{j-1}=\lambda(1-\eta)t_{j-1},~\forall j\in[N].
\end{equation}

Choose
\begin{equation}
  \label{eq:theta-a}
  \theta\in\left(0,\frac{\eta}{4}\right]\quad \text{and}\quad a\in\left(0,\frac{\eta\gamma}{4}\right]
\end{equation}
such that
\begin{displaymath}
  (1+\theta)a<1-\gamma.
\end{displaymath}
Let
\begin{displaymath}
  \mathcal G_n=\{(1-\gamma)(1+\theta)^{-k}:0\leq k\leq J_n\},
\end{displaymath}
where \(J_n\) is the least integer such that
\begin{displaymath}
  (1-\gamma)(1+\theta)^{-J_n}\leq \frac{a}{n}.
\end{displaymath}
Clearly,
\begin{displaymath}
  J_n=\left\lceil\frac{\ln((1-\gamma) n/a)}{\ln(1+\theta)}\right\rceil =O(\ln n).
\end{displaymath}
The constants \(\lambda,\eta,\alpha,\alpha',\theta,a\), and \(N\) depend only on
\(\gamma\) and \(\varrho\).

For \(x,y\in\R^n\), write \(x\leq y\) if \(x_i\leq y_i,~\forall i\in[n]\). Put
\begin{gather*}
  P(n,\gamma)=\{x\in\R^n:\ \gamma\leq\norm{x}_1\leq1,\ 0\leq x_i\leq 1-\gamma,~\forall i\in[n]\},\\
  Q(n,\gamma)=\{x\in\R^n:\ \gamma\leq\norm{x}_1\leq1,\ |x_i|\leq1-\gamma,~\forall i\in[n]\}.
\end{gather*}

For each \(i\in[n]\), let \(e_i\) be the \(i\)-th canonical basis vector of \(\mathbb{R}^n\).

\begin{lemma}
  \label{lem:common-boundary-decomposition}
  We have
  \begin{gather}
    \label{eq:common-simplex-decomp}
    \Delta_n= P(n,\gamma) \cup\gamma\Delta_n\cup \bigcup_{i=1}^n((1-\gamma) e_i+\gamma\Delta_n),\\
    \label{eq:common-cross-decomp}
    B_1^n= Q(n,\gamma)\cup\gamma B_1^n\cup \bigcup_{i=1}^n\bigl(((1-\gamma) e_i+\gamma B_1^n)\cup(-(1-\gamma) e_i+\gamma B_1^n)\bigr).
  \end{gather}
\end{lemma}
\begin{proof}
  The decomposition \eqref{eq:common-simplex-decomp} also appears in
  \cite[Eq.~(4), p.~797]{YuGaoHeWu2023}. For completeness, its direct proof is
  as follows. The right-hand side is contained in \(\Delta_n\). Conversely, if
  \(x\in\Delta_n\setminus(\gamma\Delta_n\cup P(n,\gamma))\), then
  \(\norm{x}_1>\gamma\) and \(x_i>1-\gamma\) for some \(i\in[n]\). Hence
  \(x-(1-\gamma)e_i\geq0\) and
  \(\norm{x-(1-\gamma)e_i}_1=\norm{x}_1-(1-\gamma)\leq\gamma\), which proves
  \eqref{eq:common-simplex-decomp}.

  The right-hand side of \eqref{eq:common-cross-decomp} is contained in
  \(B_1^n\). Conversely, let \(x\in B_1^n\setminus(\gamma B_1^n\cup
  Q(n,\gamma))\). Then \(\norm{x}_1>\gamma\), and \(|x_i|>1-\gamma\) for some
  \(i\in[n]\). Then
  \begin{displaymath}
    \norm{x-(1-\gamma)\operatorname{sgn}(x_i)
      e_i}_1=\norm{x}_1-(1-\gamma)\leq \gamma,
  \end{displaymath}
  or, equivalently, \( x\in(1-\gamma)\operatorname{sgn}(x_i) e_i+\gamma B_1^n\).
\end{proof}

\begin{lemma}
  \label{lem:simplex-containment}
  If \(x\in P(n,\gamma)\) and \(u\in\R_+^n\) are two points satisfying
  \begin{displaymath}
    u\leq x \quad\text{and} \quad\norm{u}_1\geq\norm{x}_1-\gamma,
  \end{displaymath}
  then \(x\in u+\gamma\Delta_n\).
\end{lemma}
\begin{proof}
  Just note that \(x-u\geq0\) and that
  \(\norm{x-u}_1=\norm{x}_1-\norm{u}_1\leq\gamma\).
\end{proof}

\subsection{Coordinate multiplicities and assignment probabilities}
\label{subsec:finite-coordinate-values}

Put
\begin{displaymath}
  Q=\left(\{0\}\cup\mathcal G_n\cup(-\mathcal G_n)\right)^n\quad\text{and}\quad
  Q^+=\left( \{0\}\cup\mathcal G_n \right)^n.
\end{displaymath}
Here \(J_n\), \(\mathcal G_n\), and the other parameters used below are those
fixed in \Cref{subsec:common-constants}.
For a finite set \(A\), denote by \(\#A\) its cardinality.

Let \(q\in Q^+\setminus\{0\}\),
\begin{displaymath}
  \beta_1>\beta_2>\cdots>\beta_L>0
\end{displaymath}
be the nonzero coordinate values of \(q\), and
\begin{displaymath}
  m_\ell=\#\{i\in[n]:\ q_i=\beta_\ell\},~\forall\ell\in[L].
\end{displaymath}
Set
\begin{displaymath}
  \mathrm{mult}(q)=(\beta_1,\ldots,\beta_L;m_1,\ldots,m_L).
\end{displaymath}
The multiplicity of \(0\) is \(n-\sum_{\ell\in[L]} m_\ell\). For convenience,
set \(\mathrm{mult}(0)=(0)\). Put
\begin{displaymath}
  S_\ell=\sum_{i\in[\ell]} m_i, \quad \sigma_\ell=\frac{S_\ell}{n},~\forall \ell\in[L], \quad\text{and}\quad \sigma_0=0.
\end{displaymath}
Then
\begin{displaymath}
  L\leq J_n+1=O(\ln n).
\end{displaymath}

\begin{lemma}
  \label{lem:multiplicity-count}
  We have
  \begin{gather*}
    \# Q^+=(J_n+2)^n=\exp\{O(n\ln\ln n)\},\\
    \# \left\{ \mathrm{mult}(q):\ q\in Q^+\right\}=\binom{n+J_n+1}{J_n+1}=e^{o(n)},\\
    \# Q=(2J_n+3)^n=\exp\{O(n\ln\ln n)\}.
  \end{gather*}
\end{lemma}
\begin{proof}
  The first and third equalities are clear. The second cardinality is precisely
  the number of nonnegative integer solutions of
  \begin{displaymath}
    y_0+y_1+\cdots+y_{J_n+1}=n.
  \end{displaymath}
  Thus
  \begin{displaymath}
    \# \left\{ \mathrm{mult}(q):\ q\in Q^+\right\}=\binom{n+J_n+1}{J_n+1}.
  \end{displaymath}
  Since \(J_n=O(\ln n)\), the standard bound
  \begin{displaymath}
    \binom mk\leq\left(\frac{em}{k}\right)^k
  \end{displaymath}
  gives
  \begin{displaymath}
    \binom{n+J_n+1}{J_n+1} \leq \left(\frac{e(n+J_n+1)}{J_n+1}\right)^{J_n+1} =\exp\{O((\ln n)^2)\}=e^{o(n)}.\qedhere
  \end{displaymath}
\end{proof}

In the following we use the convention that
\begin{displaymath}
  \sum\limits_{i\in \emptyset}y_i=0.
\end{displaymath}

\begin{lemma}
  \label{lem:coordinate-replacement}
  Let \(y\in P(n,\gamma)\). Define \(q=q(y)\in Q^+\) by
  \begin{displaymath}
    q_i=
    \begin{cases}
      0,&y_i<\flatfrac{a}{n},\\
      \max\{\beta\in\mathcal G_n:\ \beta\leq y_i\},&y_i\geq \flatfrac{a}{n}.
    \end{cases}
  \end{displaymath}
  Then
  \begin{displaymath}
    0\leq q\leq y\quad\text{and}\quad\norm{q}_1\geq\frac{\norm{y}_1-a}{1+\theta}\geq(1-\eta)\norm{y}_1.
  \end{displaymath}
\end{lemma}

\begin{proof}
  The definition gives \(0\leq q\leq y\). Put
  \begin{displaymath}
    I_1=\left\{i\in[n]:\ y_i<\frac{a}{n}\right\}\quad\text{and}\quad
    I_2=\left\{i\in[n]:\ y_i\geq \frac{a}{n}\right\}.
  \end{displaymath}
  We have
  \begin{displaymath}
    \sum\limits_{i\in I_1}y_i\leq a\quad \text{and}\quad q_i\geq
    \frac{y_i}{1+\theta},~\forall i\in I_2.
  \end{displaymath}
  Thus
  \begin{displaymath}
    \norm{q}_1=\sum_{i\in I_2}q_i \geq \frac{1}{1+\theta}\sum_{i\in I_2}y_i \geq \frac{\norm{y}_1-a}{1+\theta}.
  \end{displaymath}
  By \eqref{eq:theta-a}, we have
  \begin{displaymath}
    \frac{\norm{y}_1-a}{1+\theta} \geq \frac{1-\eta/4}{1+\eta/4}\norm{y}_1 \geq(1-\eta)\norm{y}_1.\qedhere
  \end{displaymath}
\end{proof}

Fix \(q\in Q^+\setminus\{0\}\) with
\(\mathrm{mult}(q)=(\beta_1,\ldots,\beta_L;m_1,\ldots,m_L)\), and suppose that
\begin{equation}
  \label{eq:list-mass-condition}
  d\in (0,\lambda\norm{q}_1).
\end{equation}
Put
\begin{displaymath}
  c=\frac{d}{\lambda\norm{q}_1}<1.
\end{displaymath}
For each \(\ell\in[L]\), set
\begin{displaymath}
  r_\ell=\lceil\lambda m_\ell\rceil,\quad
  R_\ell=\sum_{i\in[\ell]}r_i,\quad\text{and}\quad R_0=0.
\end{displaymath}
For each \(\ell\in[L]\), let \(\mathcal B_\ell\) be the multiset containing
\(\lfloor\lambda m_\ell\rfloor\) entries equal to \(c\beta_\ell\), and, if
\(\lambda m_\ell\notin\mathbb Z\), it contains one further entry
\begin{displaymath}
  c(\lambda m_\ell-\lfloor\lambda m_\ell\rfloor)\beta_\ell.
\end{displaymath}
Let \(A_\ell\) index the elements of \(\mathcal B_\ell\). Without loss of
generality, we may assume that elements in \(\{A_\ell : \ \ell\in [L]\}\) are
pairwise disjoint. Put \(A=\bigsqcup_{\ell\in[L]} A_\ell\), and write \(b_i\)
for the entry indexed by \(i\in A\). Since \(0<\lambda<1/2\) and \(m_\ell\geq
1\), we have
\begin{displaymath}
  r_\ell\leq m_\ell\quad\text{and}\quad R_\ell\leq S_\ell,~\forall \ell\in[L].
\end{displaymath}
Moreover,
\begin{displaymath}
  \#\mathcal B_\ell=r_\ell\quad\text{and}\quad
  \sum_{i\in A}b_i=\sum_{\ell\in[L]}c\lambda m_\ell\beta_\ell=d.
\end{displaymath}
We shall refer to the above deterministic construction of the multisets
\(\mathcal B_\ell\), their index sets \(A_\ell\), and the entries \(b_i\) from
\(q\) and \(d\) as the \emph{list construction} for \(q\) with parameter \(d\).
Here and below,
\begin{displaymath}
  (b)_r:=b(b-1)\cdots(b-r+1).
\end{displaymath}
Choose a random injection \(\iota:A\to[n]\) satisfying
\begin{displaymath}
  \mathbb{P}\{\iota=\iota_0\}=\frac{1}{(n)_{\#A}},~\text{for every injection } \iota_0:A\to[n]. 
\end{displaymath}
Define a random vector \(u\) by (cf. \Cref{fig:coordinate-assignment-event})
\begin{displaymath}
  u_i=\sum_{j\in \iota^{-1}(\{i\})}b_j,~\forall i\in[n].
\end{displaymath}
Clearly, \(\norm{u}_1=d\).

\begin{figure}[htbp]
  \centering \resizebox{\textwidth}{!}{%
    \begin{tikzpicture}[ font=\small, block/.style={draw, rounded corners=2pt,
        align=center, minimum width=2.2cm, minimum height=0.9cm},
      box/.style={draw, rounded corners=2pt, align=center, inner sep=6pt},
      arrow/.style={-{Stealth[length=2.2mm]}, thick}, every node/.style={inner
        sep=2pt} ]
      \node[box, minimum width=10cm] (multdef) at (4.3,3.3)
      {\(\mathrm{mult}(q)=(\beta_1,\ldots,\beta_L;m_1,\ldots,m_L)\), \quad
        \(\beta_1>\beta_2>\cdots>\beta_L>0\).};

      \node[box, fill=gray!4, minimum width=3.5cm, minimum height=3.6cm]
      (profile) at (0,0) {}; \node[align=center] at (0,0.45)
      {\(\begin{array}{c|c}
        \text{value} & \text{mult.}\\
        \hline
        \beta_1 & m_1\\
        \beta_2 & m_2\\
        \vdots & \vdots\\
        \beta_L & m_L\\\hline
      \end{array}\)};
    \node[align=center] at (0,-1.35)
    {\(S_\ell=m_1+\cdots+m_\ell\)};

    \node[block, fill=blue!7] (A1) at (3.8,1.45) {\(A_1\)\\ \(i\in A_1,\ b_i\)\\
      \(r_1\) elements}; \node[block, fill=green!7] (A2) at (3.8,0) {\(A_2\)\\
      \(i\in A_2,\ b_i\)\\ \(r_2\) elements}; \node[align=center] at (3.8,-0.95)
    {\(\vdots\)}; \node[block, fill=orange!10] (AL) at (3.8,-2.0) {\(A_L\)\\
      \(i\in A_L,\ b_i\)\\ \(r_L\) elements}; \node[align=center] at (3.8,2.35)
    {\(A=A_1\sqcup A_2\sqcup\cdots\sqcup A_L\)};

    \node[block, fill=blue!7, minimum width=4cm] (I1) at (8.5,1.45)
    {\(\mathcal I_1=\{k\in[n]:q_k\geq\beta_1\}\)\\
      \(\#\mathcal I_1=S_1\)}; \node[block, fill=green!7, minimum width=4cm]
    (I2) at (8.5,0)
    {\(\mathcal I_2=\{k\in[n]:q_k\geq\beta_2\}\)\\
      \(\#\mathcal I_2=S_2\)}; \node[align=center] at (8.5,-0.95) {\(\vdots\)};
    \node[block, fill=orange!10, minimum width=4cm] (IL) at (8.5,-2.0)
    {\(\mathcal I_L=\{k\in[n]:q_k\geq\beta_L\}\)\\
      \(\#\mathcal I_L=S_L\)}; \node[align=center] at (8.5,2.35) {admissible
      coordinate sets};

    \draw[arrow, blue!70!black] (A1.east) -- node[above] {\(i\mapsto\iota(i)\)}
    (I1.west); \draw[arrow, green!50!black] (A2.east) -- node[above]
    {\(i\mapsto\iota(i)\)} (I2.west); \draw[arrow, orange!80!black] (AL.east) --
    node[above] {\(i\mapsto\iota(i)\)} (IL.west);

    \node[box, minimum width=9.2cm, fill=gray!5] at (4.3,-3.35)
    {\(E(q)=\{\iota(A_\ell)\subseteq\mathcal I_\ell \text{ for every
      }\ell\in[L]\}\), where \(\iota:A\to[n]\) is injective.};
  \end{tikzpicture}
}
\caption{The event \(E(q)\) in \Cref{lem:coordinate-assignment-probability}.}
\label{fig:coordinate-assignment-event}
\end{figure}

\begin{lemma}
  \label{lem:coordinate-assignment-probability}
  Let \(E(q)\) be the event given by (cf.
  \Cref{fig:coordinate-assignment-event})
  \begin{displaymath}
    E(q):=
    \left\{\iota(A_\ell)\subseteq\{i\in[n]:q_i\geq\beta_\ell\},~\forall \ell\in[L]\right\}.
  \end{displaymath}
  Then
  \begin{equation}
    \label{eq:Pq}
    P(q):=\mathbb{P}(E(q))=\prod_{\ell\in[L]}
    \frac{(S_\ell-R_{\ell-1})_{r_\ell}}{(n-R_{\ell-1})_{r_\ell}}.
  \end{equation}
  Moreover, \(E(q)\subseteq\{u\leq q\}\).
\end{lemma}

\begin{proof}
  Order the coordinates of \(q\) by decreasing value, and reveal the random
  injection on \(A_1,\ldots,A_L\). Conditioned on the requirements for
  \(A_1,\ldots,A_{\ell-1}\), the allowed coordinates for \(A_\ell\) form a set
  of cardinality \(S_\ell-R_{\ell-1}\). Hence the corresponding conditional
  probability is
  \begin{displaymath}
    \frac{(S_\ell-R_{\ell-1})_{r_\ell}}{(n-R_{\ell-1})_{r_\ell}}.
  \end{displaymath}
  Multiplication over \(\ell\in[L]\) gives \eqref{eq:Pq}. Assume that
  \(E(q)\) holds. If \(k\in\iota(A_\ell)\), then the injectivity of \(\iota\)
  gives a unique \(j=\iota^{-1}(k)\in A_\ell\), and therefore
  \begin{displaymath}
    u_k=b_j\leq c\beta_\ell<\beta_\ell\leq q_k.
  \end{displaymath}
  If \(k\notin\iota(A)\), then \(u_k=0\leq q_k\). Thus \(u\leq q\), as
  required.
\end{proof}

For fixed \(q\), \(\mathrm{mult}(q)\), and \(d\),
\Cref{lem:coordinate-assignment-probability} gives
\begin{displaymath}
  \mathbb{P}\{u\leq q,\ \norm{u}_1=d\}\geq P(q).
\end{displaymath}

\begin{lemma}
  \label{lem:integral-exponent}
  Put
  \begin{displaymath}
    \Theta_n=\left\lceil\frac{2(J_n+1)}{1-\lambda}\right\rceil
  \end{displaymath}
  and set
  \begin{align*}
    \varepsilon_n=\frac{1}{n}\bigg[
    &(\lambda\Theta_n+J_n+1)\ln n+\lambda\Theta_n\ln\frac{n}{1-\lambda}
      +(J_n+1)\ln n\\
    &+\frac{4\lambda(J_n+1)}{1-\lambda}(1+\ln n)
      \bigg].
  \end{align*}
  Then \(\varepsilon_n\to0\) and
  \begin{displaymath}
    -\frac{1}{n}\ln P(q)\leq-\ln(1-\lambda)+\varepsilon_n,~\forall q\in Q^+\setminus\{0\}.
  \end{displaymath}
\end{lemma}

\begin{proof}
  Fix \(q\in Q^+\setminus\{0\}\). Define
  \begin{displaymath}
    G_\ell=\ln\frac{n-R_\ell+1}{S_\ell-R_\ell+1},~\forall
    \ell\in[L]\quad\text{and}\quad A(q)=\sum_{\ell\in[L]} r_\ell G_\ell.
  \end{displaymath}

  Applying the elementary inequality
  \begin{displaymath}
    \frac{(s)_r}{(t)_r}\geq \left(\frac{s-r+1}{t-r+1}\right)^r,~\forall
    r,s,t\in\mathbb Z\text{ with }0\leq r\leq s\leq t,
  \end{displaymath}
  to \eqref{eq:Pq} gives
  \begin{equation}
    \label{eq:logP-explicit}
    -\ln P(q)\leq A(q).
  \end{equation}
  Define
  \begin{displaymath}
    H_\ell=\ln\frac{n-\lambda S_\ell}{(1-\lambda)S_\ell},~\forall
    \ell\in[L]\quad\text{and}\quad B(q)=\sum_{\ell\in[L]}m_\ell H_\ell.
  \end{displaymath}

  Since \(r_\ell=\lceil\lambda m_\ell\rceil\) and \(L\leq J_n+1\), we have
  \begin{displaymath}
    0\leq r_\ell-\lambda m_\ell<1\quad\text{and}\quad
    0\leq R_\ell-\lambda S_\ell\leq J_n+1,~\forall \ell\in[L].
  \end{displaymath}

  Let
  \begin{displaymath}
    \ell_0=\max\{\ell\in[L]:S_\ell<\Theta_n\},
  \end{displaymath}
  with the convention \(\max\emptyset=0\). For every \(\ell\in[L]\), we have
  \begin{displaymath}
    0\leq G_\ell\leq\ln n\quad\text{and}\quad
    0\leq H_\ell\leq\ln\frac{n}{1-\lambda}.
  \end{displaymath}
  Therefore
  \begin{equation}
    \label{eq:small-A}
    \sum_{\ell\in[\ell_0]}r_\ell G_\ell \leq R_{\ell_0}\ln n \leq(\lambda\Theta_n+J_n+1)\ln n
  \end{equation}
  and
  \begin{equation}
    \label{eq:small-B}
    \sum_{\ell\in[\ell_0]}m_\ell H_\ell\leq \Theta_n\ln\frac{n}{1-\lambda}.
  \end{equation}
  Also,
  \begin{equation}
    \label{eq:rounding-G}
    \sum_{\ell=\ell_0+1}^{L}|r_\ell-\lambda m_\ell|G_\ell \leq (J_n+1)\ln n.
  \end{equation}

  For \(\ell\in[\ell_0+1,L]\cap \mathbb{Z}\), set \(E_\ell=R_\ell-\lambda
  S_\ell\). Since \(S_\ell\geq\Theta_n\), we have
  \begin{displaymath}
    \left|\frac{E_\ell-1}{n-\lambda S_\ell}\right|
    \leq
    \left|\frac{E_\ell-1}{(1-\lambda)S_\ell}\right|
    \leq \frac{J_n+1}{(1-\lambda)\Theta_n}\leq\frac12.
  \end{displaymath}
  Since \(|\ln(1+z)|\leq2|z|\) when \(|z|\leq 1/2\), we have
  \begin{displaymath}
    \left|\ln\frac{n-R_\ell+1}{n-\lambda S_\ell}\right|
    \leq \frac{2(J_n+1)}{(1-\lambda)S_\ell}\quad\text{and}\quad \left|\ln\frac{S_\ell-R_\ell+1}{(1-\lambda)S_\ell}\right|
    \leq \frac{2(J_n+1)}{(1-\lambda)S_\ell}.
  \end{displaymath}
  Hence
  \begin{equation}
    \label{eq:G-H-explicit}
    |G_\ell-H_\ell|\leq\frac{4(J_n+1)}{(1-\lambda)S_\ell},~\forall \ell\in[\ell_0+1,L]\cap \mathbb{Z}.
  \end{equation}
  Moreover,
  \begin{displaymath}
    \sum_{\ell\in [L]}\frac{m_\ell}{S_\ell}
    =1+\sum_{\ell=2}^{L}\frac{S_\ell-S_{\ell-1}}{S_\ell}
    \leq1+\int_{S_1}^{S_L}\frac{{\rm d}t}{t}
    \leq1+\ln n.
  \end{displaymath}
  By \eqref{eq:G-H-explicit},
  \begin{equation}
    \label{eq:G-H-sum-explicit}
    \sum_{\ell=\ell_0+1}^{L}m_\ell |G_\ell-H_\ell|
    \leq \frac{4(J_n+1)}{1-\lambda}(1+\ln n).
  \end{equation}

  By \eqref{eq:small-A}--\eqref{eq:G-H-sum-explicit},
  \begin{displaymath}
    |A(q)-\lambda B(q)|\leq\varepsilon_n n.
  \end{displaymath}
  By \eqref{eq:logP-explicit},
  \begin{displaymath}
    -\ln P(q)\leq\lambda B(q)+\varepsilon_n n.
  \end{displaymath}

  Clearly, \(m_\ell=n(\sigma_\ell-\sigma_{\ell-1})\), \(\forall\ell\in[L]\).
  Therefore,
  \begin{displaymath}
    B(q)= n\sum_{\ell\in[L]}(\sigma_\ell-\sigma_{\ell-1})\ln\frac{1-\lambda\sigma_\ell}{(1-\lambda)\sigma_\ell}.
  \end{displaymath}
  Hence
  \begin{displaymath}
    -\frac1n\ln P(q)\leq
    \lambda\sum_{\ell\in[L]}(\sigma_\ell-\sigma_{\ell-1})
    \ln\frac{1-\lambda\sigma_\ell}{(1-\lambda)\sigma_\ell}
    +\varepsilon_n.
  \end{displaymath}

  Let
  \begin{displaymath}
    \psi(s)=\ln\frac{1-\lambda s}{(1-\lambda)s},~\forall 0<s\leq 1.
  \end{displaymath}
  Then \(\psi\) is nonnegative and decreasing. Hence
  \begin{displaymath}
    \lambda\sum_{\ell\in[L]}(\sigma_\ell-\sigma_{\ell-1})\psi(\sigma_\ell)
    \leq \lambda\int_0^1\psi(s)\,{\rm d}s=-\ln(1-\lambda).
  \end{displaymath}
  Since \(J_n=O(\ln n)\), we have \(\Theta_n=O(\ln n)\), and hence
  \(\varepsilon_n=O((\ln n)^2/n)\to0\). Combining the last two estimates proves
  the lemma.
\end{proof}

\begin{lemma}
  \label{lem:crosspolytope-probability}
  For a fixed \(z\in Q\setminus\{0\}\), let \(q=(|z_1|,\ldots,|z_n|)\in Q^+\)
  and \(d\) be a number satisfying \eqref{eq:list-mass-condition}. Apply the
  list construction to \(q\) with \(d\) and inject the list into \([n]\) and
  assign independent random signs to the occupied coordinates. Denote by
  \(P_{\rm sgn}(z)\) the probability of the event that \(E(q)\) holds and all
  signs of the occupied coordinates agree with \(z\). We have, with the
  \(\varepsilon_n\) defined in \Cref{lem:integral-exponent},
  \begin{displaymath}
    -\ln P_{\rm sgn}(z)
    \leq\left(-\ln(1-\lambda)+\lambda\ln2+\varepsilon_n+\frac{(J_n+1)\ln 2}{n}\right)n,~\forall z\in Q\setminus\{0\}.
  \end{displaymath}
\end{lemma}
\begin{proof}
  Clearly,
  \begin{displaymath}
    P_{\rm sgn}(z)=2^{-R_L}P(q).
  \end{displaymath}
  By \Cref{lem:integral-exponent},
  \begin{displaymath}
    -\ln P(q)\leq(-\ln(1-\lambda)+\varepsilon_n)n.
  \end{displaymath}
  Moreover \(R_L\leq\lambda S_L+L\), \(S_L\leq n\), and \(L\leq J_n+1\). Thus we
  have
  \begin{displaymath}
    R_L\ln2\leq\left(\lambda\ln2+\frac{(J_n+1)\ln 2}{n}\right)n,
  \end{displaymath}
  and the desired inequality follows.
\end{proof}

\subsection{The estimate for \(\Delta_n\)}
\label{subsec:simplex-asymptotic}

Take \(\varrho=1/2\). 
\begin{proof}[Proof of \Cref{thm:simplex-half}]
  Since \(\Gamma_{2^n}(\Delta_n)\geq 1/2\), we only need to show that, for fixed
  \(\gamma\in(1/2,1)\), \(\Gamma_{2^n}(\Delta_n)\leq \gamma\) whenever \(n\) is
  sufficiently large. It suffices to show that, for sufficiently large \(n\),
  \(P(n,\gamma)\) can be covered by \(2^n-(n+1)\) translates of \(\gamma
  \Delta_n\). Choose \(\eps\in(0,\ln
  2+\ln(1-\lambda))\).

  For each \(j\in[N]\), set
  \begin{displaymath}
    P_j=\{x\in P(n,\gamma):\ t_{j-1}\leq\norm{x}_1\leq t_j\}.
  \end{displaymath}
  For \(x\in P_j\), let \(q(x)\) be given by \Cref{lem:coordinate-replacement}.
  Then
  \begin{displaymath}
    \norm{q(x)}_1\geq(1-\eta)\norm{x}_1\geq(1-\eta)t_{j-1}.
  \end{displaymath}
  By \eqref{eq:common-layer-ineq},
  \begin{displaymath}
    d_j<\lambda(1-\eta)t_{j-1}\leq\lambda\norm{q(x)}_1.
  \end{displaymath}
  Thus the construction in \Cref{subsec:finite-coordinate-values} applies to
  \(q(x)\) with \(d=d_j\). If \(E(q(x))\) holds, the resulting vector \(u\)
  satisfies
  \begin{displaymath}
    u\leq q(x)\leq x\quad\text{and}\quad \norm{u}_1=d_j=t_j-\gamma.
  \end{displaymath}
  Since \(\norm{x}_1\leq t_j\), \Cref{lem:simplex-containment} gives \(x\in
  u+\gamma\Delta_n\).

  Put
  \begin{displaymath}
    \mathcal M_j=\{\mathrm{mult}(q(x)):\ x\in P_j\},~\forall j\in[N],
  \end{displaymath}
  and, for each \(\mu\in\mathcal M_j\),
  \begin{displaymath}
    Q_{j,\mu}=\{q(x):\ x\in P_j,\ \mathrm{mult}(q(x))=\mu\}.
  \end{displaymath}
  For every \(j\in[N]\) and every \(\mu\in\mathcal M_j\), take
  \begin{displaymath}
    M_n=\left\lceil\exp\{(-\ln(1-\lambda)+\eps)n\}\right\rceil
  \end{displaymath}
  independent copies of the random injection constructed from \(\mu\) with
  \(d=d_j\). For \(q\in Q_{j,\mu}\), let \(F_{j,\mu,q}\) be the event that none
  of these \(M_n\) injections lies in \(E(q)\). Then, by using the estimate
  \begin{equation}
    \label{eq:probability-estimate}
    (1-t)^m\leq e^{-mt},~\forall t\in (0,1),~\forall m\in [1,\infty)\cap \mathbb{Z},
  \end{equation}
  we have
  \begin{displaymath}
    \mathbb{P}(F_{j,\mu,q})\leq(1-P(q))^{M_n}\leq\exp\{-M_nP(q)\}.
  \end{displaymath}
  By \Cref{lem:integral-exponent}, \(\varepsilon_n\) is a sequence independent
  of \(q\) and converging to \(0\), such that
  \(P(q)\geq\exp\{-(-\ln(1-\lambda)+\varepsilon_n)n\}\). Hence, for all
  sufficiently large \(n\),
  \begin{displaymath}
    M_nP(q)\geq\exp\left(\frac{\eps n}{2}\right),~\forall q\in Q_{j,\mu}.
  \end{displaymath}
  By \Cref{lem:multiplicity-count},
  \begin{displaymath}
    \#\{(j,\mu):\ j\in[N],\ \mu\in\mathcal M_j\}\leq N
    e^{o(n)}\quad\text{and}\quad \#Q^+=\exp\{O(n\ln\ln n)\}.
  \end{displaymath}
  Therefore
  \begin{equation}
    \label{eq:union-probability}
    \mathbb{P}\left(\bigcup_{j\in[N]}\bigcup_{\mu\in\mathcal M_j}\bigcup_{q\in Q_{j,\mu}}F_{j,\mu,q}\right)\leq\exp\left\{O(n\ln\ln n)-\exp\left(\frac{\eps n}{2}\right)\right\}<1
  \end{equation}
  for all sufficiently large \(n\). Thus there is a deterministic choice \(C\)
  of centers which works for every \(q\in Q_{j,\mu}\), every \(\mu\in\mathcal
  M_j\), and every \(j\in[N]\). The first part of the proof shows that
  \(C+\gamma \Delta_n\) covers \(P(n,\gamma)\).

  The cardinality of \(C\) is at most
  \begin{align*}
    N e^{o(n)}\left( \exp\{(-\ln(1-\lambda)+\eps)n\}+1 \right)&=\exp\{(-\ln(1-\lambda)+\eps+o(1))n\}\\
                                                              &<2^n-n-1
  \end{align*}
  for all sufficiently large \(n\).
\end{proof}

\begin{remark}
  The proof yields the following slightly more precise statement. Fix
  \(\upsilon\in(1/2,1)\) and \(R\in (-\ln \upsilon,\ln 2)\). Then, for all sufficiently large
  \(n\), there exists \(C_n\subseteq\mathbb R_+^n\) such that
  \begin{displaymath}
    \Delta_n\subseteq C_n+\upsilon\Delta_n \quad\text{and}\quad \#C_n\leq e^{Rn}.
  \end{displaymath}
  Indeed, apply the preceding proof with \(\gamma=\upsilon\), choosing
  \(\lambda\in(1-\upsilon,1/2)\) and \(\varepsilon>0\) so that
  \[
    -\ln(1-\lambda)+\varepsilon<R,
  \]
  and then use \eqref{eq:common-simplex-decomp}.
\end{remark}

\subsection{The estimate for \(B_1^n\)}
\label{subsec:cross-asymptotic}

\begin{proof}[Proof of \Cref{thm:b1-asymptotic}]
  Take \(\varrho=\gamma_{\rm sgn}\) and fix \(\gamma\in(\varrho,1)\). It
  suffices to show that, for sufficiently large \(n\), \(Q(n,\gamma)\) can be
  covered by \(2^n-2n-1\) translates of \(\gamma B_1^n\).

  For each \(j\in[N]\), set
  \begin{displaymath}
    Q_j=\{x\in Q(n,\gamma):\ t_{j-1}\leq\norm{x}_1\leq t_j\}.
  \end{displaymath}
  For \(x\in Q_j\), apply \Cref{lem:coordinate-replacement} to
  \((|x_1|,\ldots,|x_n|)\), and denote the resulting point of \(Q^+\) by
  \(q(x)\). Then
  \begin{displaymath}
    0\leq q(x)_i\leq |x_i|,~\forall i\in[n]\quad\text{and}\quad \norm{q(x)}_1\geq(1-\eta)\norm{x}_1.
  \end{displaymath}
  Let \(q^{\rm sgn}(x)\) be the point defined by
  \begin{displaymath}
    q^{\rm sgn}(x)_i=
    \begin{cases}
      \operatorname{sgn}(x_i)q(x)_i,&q(x)_i>0,\\
      0,&q(x)_i=0.
    \end{cases}
  \end{displaymath}
  By \eqref{eq:common-layer-ineq},
  \begin{displaymath}
    d_j<\lambda(1-\eta)t_{j-1}\leq\lambda\norm{q(x)}_1,
  \end{displaymath}
  so the list construction in \Cref{subsec:finite-coordinate-values} applies to
  \(q(x)\) with \(d=d_j\).

  Inject the list into \([n]\) as in \Cref{subsec:finite-coordinate-values} and
  assign independent random signs to the occupied coordinates. When \(E(q(x))\)
  holds and all signs of the occupied coordinates agree with \(q^{\rm sgn}(x)\),
  the resulting vector \(u\) satisfies
  \begin{displaymath}
    |u_i|\leq q(x)_i\leq |x_i|,\quad \operatorname{sgn}(u_i)=\operatorname{sgn}(x_i)\ \text{when }u_i\neq0, \quad\text{and}\quad  \norm{u}_1=d_j.
  \end{displaymath}
  Hence
  \begin{displaymath}
    \norm{x-u}_1=\norm{x}_1-\norm{u}_1\leq t_j-(t_j-\gamma)=\gamma,
  \end{displaymath}
  and \(x\in u+\gamma B_1^n\).
  
  Since \(\lambda<1-\varrho\), the defining equation of \(\varrho=\gamma_{\rm
    sgn}\) gives
  \begin{displaymath}
    \rho_{\rm sgn}(\lambda):=-\ln(1-\lambda)+\lambda\ln 2<\ln 2.
  \end{displaymath}
  Choose \(\eps\in (0,\ln 2-\rho_{\rm sgn}(\lambda))\). Put
  \begin{gather*}
    \mathcal M_j=\{\mathrm{mult}(q(x)):\ x\in Q_j\},~\forall j\in[N],\\
    Q'_{j,\mu}=\{q^{\rm sgn}(x):\ x\in Q_j,\ \mathrm{mult}(q(x))=\mu\},~\forall j\in[N],~\forall \mu\in\mathcal M_j,
  \end{gather*}  
  and take
  \begin{displaymath}
    M_n=\left\lceil\exp\{(\rho_{\rm sgn}(\lambda)+\eps)n\}\right\rceil
  \end{displaymath}
  independent copies of the random vector obtained from \(\mu\) with \(d=d_j\),
  together with independent signs on the occupied coordinates; denote them by
  \(u^{(i)}_{j,\mu},~\forall i\in[M_n]\). For each \(z\in Q'_{j,\mu}\), let
  \(E'_{j,\mu,z}\) be the event that none of the corresponding \(M_n\) random
  vectors satisfies both \(E((|z_1|,\ldots,|z_n|))\) and the required sign
  agreement with \(z\). Then, using \eqref{eq:probability-estimate} again, we
  have
  \begin{displaymath}
    \mathbb{P}(E'_{j,\mu,z}) \leq (1-P_{\rm sgn}(z))^{M_{n}} \leq \exp\{-M_{n}P_{\rm sgn}(z)\}.
  \end{displaymath}
  Set
  \begin{displaymath}
    \varepsilon_n^{(1)}=\varepsilon_n+\frac{(J_n+1)\ln 2}{n}.
  \end{displaymath}
  By \Cref{lem:crosspolytope-probability},
  \begin{align*}
    M_{n}P_{\rm sgn}(z) &\geq \exp\{(\rho_{\rm sgn}(\lambda)+\eps)n\}\cdot \exp\left\{-\left(\rho_{\rm sgn}(\lambda)+\varepsilon_n^{(1)}\right)n\right\}\\
                        &=\exp\left\{\left(\eps-\varepsilon_n^{(1)}\right)n\right\}.
  \end{align*}
  This, together with \Cref{lem:multiplicity-count} implies that
  \begin{align*}
    &\mathbb{P}\left(\bigcup_{j\in[N]}\bigcup_{\mu\in \mathcal{M}_{j}}\bigcup_{z\in Q'_{j,\mu}} E'_{j,\mu,z}\right)\\
    \leq& N\cdot e^{o(n)}\cdot\exp\{O(n\ln\ln n)\}\cdot \exp\left\{-\exp \left\{\left(\eps - \varepsilon^{(1)}_{n}\right)n\right\}\right\}<1
  \end{align*}
  holds for all sufficiently large \(n\). We have
  \begin{align*}
    \#\{u^{(i)}_{j,\mu}:\ j\in[N],\ \mu\in\mathcal M_j,\ i\in[M_n]\}
    &\leq N e^{o(n)}\left( \exp\{(\rho_{\rm sgn}(\lambda)+\eps)n\}+1 \right)\\
    &=\exp\{(\rho_{\rm sgn}(\lambda)+\eps+o(1))n\}\\
    &<2^n-2n-1
  \end{align*}
  holds for all sufficiently large \(n\). Since the probability above is less
  than one, there is a deterministic choice
  \begin{displaymath}
    \left\{ u^{(i,0)}_{j,\mu} :\ j\in[N],\mu\in \mathcal{M}_j,i\in[M_n] \right\}
  \end{displaymath}
  such that for
  every \(j\in[N]\), every \(\mu\in\mathcal M_j\), and every \(z\in
  Q'_{j,\mu}\), at least one \(u^{(i_0,0)}_{j,\mu}\) satisfies the two
  requirements.

  Let \(x\in Q(n,\gamma)\). Suppose that \(x\in Q_j\). Let
  \begin{displaymath}
    z=q^{\rm sgn}(x)\quad \text{and} \quad
    \mu=\mathrm{mult}((|z_1|,\ldots,|z_n|))\in\mathcal M_j. 
  \end{displaymath}
  Since \(z\in Q'_{j,\mu}\), we have \(x\in u^{(i,0)}_{j,\mu}+\gamma
  B_1^n\) for some \(i\in[M_n]\). This completes the proof.
\end{proof}

\begin{remark}
  The proof yields the following slightly more precise statement. Fix $\nu\in
  (\gamma_{\rm{sgn}},1)$ and $R\in (-\ln \nu+(1-\nu)\ln 2,\ln 2)$. Then for all
  sufficiently large $n$, there exists $U_n\subset \mathbb{R}^n$ with $\#
  U_n\leq \exp(Rn)$ such that, for every $y\in B_1^n$, there exists $u\in U_n$
  satisfying
  \begin{displaymath}
    u_iy_i\geq 0, |u_i|\leq |y_i|,\quad\forall i\in[n]\quad \text{and}\quad
    \sum\limits_{i=1}^n(|y_i|-|u_i|)\leq \nu.
  \end{displaymath}
  Indeed, apply the preceding proof with $\gamma=\nu$, choosing $\lambda \in
  (1-\nu,1-\gamma_{\rm{sgn}})$ and $\varepsilon>0$ so that
  $\rho_{\rm{sgn}}(\lambda)+\varepsilon<R$. To treat the points outside
  $Q(n,\nu)$, add the centers $0$ and $\pm(1-\nu)e_i$ for $i\in[n]$. If
  $\|y\|_1\leq\nu$, take $u=0$. Otherwise, as in the proof of
  \eqref{eq:common-cross-decomp}, choose an index $i$ with $|y_i|>1-\nu$ and
  take $u=(1-\nu)\operatorname{sgn}(y_i)e_i$. This center satisfies
  $u_jy_j\geq 0$ and $|u_j|\leq|y_j|$ for every $j\in[n]$, and
  \begin{displaymath}
    \sum_{j=1}^n (|y_j|-|u_j|)=\|y\|_1-(1-\nu)\leq\nu.
  \end{displaymath}
  Since $\rho_{\mathrm{sgn}}(\lambda)+\varepsilon<R$, adding these $2n+1$
  centers still leaves at most $\exp(Rn)$ centers for all sufficiently large
  $n$.
\end{remark}

\section{Dimension-free estimates for \(B_p^n\)}

This section proves \Cref{thm:b1-main,prop:fixed-p} by considering four
disjoint cases for the dimension. If \(3\leq n\leq178\) and \(n\neq6\),
\Cref{lem:b1-finite}, verified in Appendix
\ref{app:finite-verification}, gives
\(M(n,\lceil n/5\rceil)\leq2^n\). Applying \Cref{lem:lp-lattice} with
\(k=\lceil n/5\rceil\) yields
\(\Gamma_{2^n}(B_p^n)\leq(5/6)^{1/p}\). For \(n=6\), the monotonicity of
\(\Gamma_m\) in \(m\) and \Cref{lem:two-n} give
\(\Gamma_{64}(B_p^6)\leq\Gamma_{12}(B_p^6)\leq(5/6)^{1/p}\). For
\(n\geq179\), \eqref{eq:vartheta} gives
\(\lfloor\vartheta_2n\rfloor\geq n/5\), and \Cref{lem:lmw} yields
\(\Gamma_{2^n}(B_p^n)\leq(5/6)^{1/p}\). For \(n=2\), one has
\(\Gamma_4(B_1^2)=1/2\), while Lassak's planar estimate gives
\(\Gamma_4(B_p^2)\leq2^{-1/2}<(5/6)^{1/p}\) for every
\(p\in[1,\infty)\). Specializing these estimates to \(p=1\) proves
\Cref{thm:b1-main}; keeping \(p\in[1,\infty)\) arbitrary proves
\Cref{prop:fixed-p}. \Cref{sec:uniform-p-modular} combines
\Cref{prop:fixed-p} with further estimates to obtain an upper bound independent
of both \(n\) and \(p\).

For \(n,k\in\N\), put
\begin{displaymath}
  M(n,k)=1+\sum_{i\in [\min\{n,k\}]}2^i\binom{n}{i}\binom{k-1}{i-1}.
\end{displaymath}

\begin{lemma}[{\cite[Lemma 2.1 and Corollary 2.2]{ChenGaoLiWu2025}}]
  \label{lem:lp-lattice}
  Let \(n\geq3\) and \(1\leq k\leq n/2\). If \(M(n,k)\leq2^n\), then, for every
  \(p\geq1\),
  \begin{displaymath}
    \Gamma_{2^n}(B_p^n)\leq \left(\frac{n}{n+k}\right)^{1/p}.
  \end{displaymath}
\end{lemma}

\begin{lemma}
  \label{lem:b1-finite}
  For every integer \(n\in[3,178]\) with \(n\neq6\),
  \begin{displaymath}
    M\left(n,\left\lceil\frac{n}{5}\right\rceil\right)\leq 2^n.
  \end{displaymath}
\end{lemma}
\begin{proof}
  The desired inequality can be checked by the code in Appendix \ref{app:finite-verification}.
\end{proof}

\begin{lemma}
  \label{lem:lmw}
  Let \(\vartheta_2\) be the positive solution of
  \begin{displaymath}
    g(x):=\frac{2^x(1+x)^{1+x}}{x^x}=2.
  \end{displaymath}
  Then, for every \(n\geq3\) and every \(p\geq1\),
  \begin{displaymath}
    \Gamma_{2^n}(B_p^n)\leq\left(\frac{n}{n+\lfloor \vartheta_2n\rfloor}\right)^{1/p}.
  \end{displaymath}
  Moreover, \(\vartheta_2>0.20559\).
\end{lemma}
\begin{proof}
  The upper bound for \(\Gamma_{2^n}(B_p^n)\) can be found in \cite[Proposition
  5]{LiMengWu2022}.

  Clearly,
  \begin{displaymath}
    \frac{d}{dx}\ln g(x)=\ln2+\ln(1+x)-\ln x>0,~\forall x>0.
  \end{displaymath}
  Direct calculation gives
  \begin{displaymath}
    g(0.20559)=1.9999638639\cdots<2,
  \end{displaymath}
  and therefore \(\vartheta_2>0.20559\).
\end{proof}

\begin{lemma}[{\cite[Lemma 3.1]{ChenGaoLiWu2025}}]
  \label{lem:two-n}
  For every \(n\geq2\) and every \(p\geq1\),
  \begin{displaymath}
    \Gamma_{2n}(B_p^n)\leq \left(1-\frac{1}{n}\right)^{1/p}.
  \end{displaymath}
\end{lemma}

\begin{proof}[Proof of \Cref{thm:b1-main}]
  Clearly, \(\Gamma_4(B_1^2)=\flatfrac{1}{2}\).

  When \(3\leq n\leq178\) and \(n\neq6\), \Cref{lem:b1-finite} and
  \Cref{lem:lp-lattice} imply that
  \begin{displaymath}
    \Gamma_{2^n}(B_1^n)\leq\frac{n}{n+\lceil n/5\rceil}\leq \frac{5}{6}.
  \end{displaymath}
  If \(n=6\), then \Cref{lem:two-n} gives
  \begin{displaymath}
    \Gamma_{64}(B_1^6)\leq \Gamma_{12}(B_1^6)\leq\frac{5}{6}.
  \end{displaymath}
  By \Cref{lem:lmw} and
  \begin{equation}
    \label{eq:vartheta}
    \lfloor \vartheta_2n\rfloor \geq \vartheta_2n-1>0.20559 n-1\geq
    \frac{n}{5},~\forall n\geq 179,
  \end{equation}
  we have
  \begin{displaymath}
    \Gamma_{2^n}(B_1^n)\leq\frac{n}{n+\lfloor \vartheta_2n\rfloor}\leq \frac{5}{6},~\forall
    n\geq 179.\qedhere
  \end{displaymath}
\end{proof}

\begin{proof}[Proof of \Cref{prop:fixed-p}]
  Lassak \cite{Lassak1986} showed that \(\Gamma_4(K)\leq2^{-1/2}\) holds for
  each planar convex body. Thus,
  \begin{displaymath}
    \Gamma_4(B_p^2)\leq2^{-1/2}<\left(\frac{5}{6}\right)^{1/p}.
  \end{displaymath}
  For \(3\leq n\leq178\), \(n\neq6\), \Cref{lem:lp-lattice} gives
  \begin{displaymath}
    \Gamma_{2^n}(B_p^n)\leq\left(\frac{n}{n+\lceil n/5\rceil}\right)^{1/p} \leq \left(\frac{5}{6}\right)^{1/p}.
  \end{displaymath}
  When \(n=6\), \Cref{lem:two-n} implies that
  \begin{displaymath}
    \Gamma_{64}(B_p^6)\leq\Gamma_{12}(B_p^6)\leq \left(\frac{5}{6}\right)^{1/p}.
  \end{displaymath}
  \Cref{lem:lmw} and \eqref{eq:vartheta} show that
  \begin{displaymath}
    \Gamma_{2^n}(B_p^n)\leq\left(\frac{5}{6}\right)^{1/p},~\forall n\geq 179.\qedhere
  \end{displaymath}
\end{proof}

\section{Uniform estimates for \(B_p^n\)}
\label{sec:uniform-p-modular}

The proof of \Cref{thm:uniform-p} first establishes
\Cref{prop:modular-middle-p-cover} for \(n\geq N_*\) and
\(p\in[\pi_0,Q_n(\beta)]\), where \(Q_n(\beta)\) is defined below. For
\(x\in B_p^n\), \Cref{lem:modular-sparse-coordinates} gives
\(S_R(x)\in\mathcal S_R(n,p)\). By \Cref{lem:modular-lifting}, it suffices,
for each \(S\in\mathcal S_R(n,p)\), to choose a \(\nu\)-net
\(\mathcal W_S\) in \(B_p^S\), an \(\eta n^{-1/p}\)-net \(\mathcal Y_S\)
in \(B_p^{S^c}\), and a finite set \(\mathcal V_S\subseteq B_p^{S^c}\) such
that every \(y\in\mathcal Y_S\cap B_p^n(M)\) lies within \(\chi\) of some
\(v\in\mathcal V_S\). Equation \eqref{eq:cardinality-YS} bounds the
cardinalities of the two nets. With \(K_{n,p}\) defined by \eqref{eq:Knq},
\eqref{eq:sum-card-bound} gives
\[
  \sum_{S\in\mathcal S_R(n,p)}\#\mathcal W_S\leq K_{n,p},
\]
and \Cref{claim:modular-support-budget} shows that \(K_{n,p}\leq2^{n-1}\).
For each \(S\), \Cref{claim:modular-one-point-hit} applies
\Cref{lem:modular-sign-hit} on \(S^c\) and gives a lower bound \(q_{n,p}\),
uniform over \(y\in\mathcal Y_S\cap B_p^n(M)\), for the probability that one
random vector \(v^{(S,i)}\) satisfies \(\norm{y-v^{(S,i)}}_p\leq\chi\).
Writing \(d=n-\#S\), \Cref{claim:modular-dimension-transfer} gives
\(p\leq\ln d/\ln\beta_0\), while
\Cref{claim:modular-xi-large}, using \eqref{eq:lower-bound-xi}, gives
\(d\xi_p\geq2\).

Set \(L_{n,p}=\lfloor2^n/K_{n,p}\rfloor\). Combining
\Cref{claim:modular-support-budget} with \eqref{eq:lower-bound-xi},
\Cref{claim:modular-trial-budget} gives the required lower bound for
\(L_{n,p}q_{n,p}\). Then \Cref{claim:modular-simultaneous-cover} combines
\Cref{claim:modular-one-point-hit} and \Cref{claim:modular-trial-budget} with the
cardinality bounds in \Cref{lem:modular-net-budget} and
\eqref{eq:cardinality-YS}, and uses \eqref{eq:probability-estimate} to apply a
union bound over all \(S\in\mathcal S_R(n,p)\) and all
\(y\in\mathcal Y_S\cap B_p^n(M)\). The resulting deterministic sets
\(\mathcal V_S=\{v^{(S,i)}:i\in[L_{n,p}]\}\) give at most
\(K_{n,p}L_{n,p}\leq2^n\) centers in \Cref{lem:modular-lifting}. That lemma
therefore gives the covering ratio
\(\bigl((\chi+\eta)^p+\nu^p\bigr)^{1/p}\leq\kappa=91/100\).

To deduce \Cref{thm:uniform-p}, enlarge \(N_*\), if necessary, and set
\(P=\ln N_*/\ln\beta\) and \(\kappa_*=(5/6)^{1/P}\). The identity
\(\Gamma_{2^n}(B_\infty^n)=1/2\) treats \(p=\infty\). For finite \(p\), the
remaining parameter space is divided into four cases: \(p\leq P\), where
\Cref{prop:fixed-p} applies; \(p>P\) and \(n<N_*\), where
\eqref{eq:cube-cover} applies; \(p>P\), \(n\geq N_*\), and
\(p\geq Q_n(\beta)\), where \eqref{eq:cube-cover} applies again; and
\(p>P\), \(n\geq N_*\), and \(p<Q_n(\beta)\). In the last case,
\(P\geq\pi_0\) implies \(p\in[\pi_0,Q_n(\beta)]\), so
\Cref{prop:modular-middle-p-cover} applies. These cases cover every
\(n\geq2\) and every \(p\in[1,\infty]\).

\smallskip

For \(\beta>1\), put
\begin{displaymath}
  Q_n(\beta)=\frac{\ln n}{\ln\beta}.
\end{displaymath}
For \(M>0\), set
\begin{displaymath}
  B_p^n(M)=\{x\in B_p^n:|x_i|\leq Mn^{-1/p},~\forall i\in[n]\}.
\end{displaymath}

A Rademacher random variable takes the values \(-1\) and \(1\), each with
probability \(1/2\).

For \(S\subseteq[n]\), we shall use the notation
\begin{displaymath}
  B_p^S=\{z\in B_p^n:z_i=0,~\forall i\notin S\}\quad\text{and}\quad B_p^{S^c}=\{z\in B_p^n:z_i=0,~\forall i\in S\}.
\end{displaymath}
For \(x\in\mathbb R^n\), let \(x_S\) and \(x_{S^c}\) be defined by
\begin{displaymath}
  (x_S)_i=
  \begin{cases}
    x_i,&i\in S,\\
    0,&i\not\in S,
  \end{cases}\quad \text{and}\quad   (x_{S^c})_i=
  \begin{cases}
    0,&i\in S,\\
    x_i,&i\not\in S.
  \end{cases}
\end{displaymath}
For \(p\in[1,\infty)\) and \(R>1\), set
\begin{displaymath}
  S_R(x)=\{j\in[n]: |x_j|n^{1/p}>R\},~\forall x\in B_p^n
\end{displaymath}
and
\begin{displaymath}
  \mathcal S_R(n,p)=\{S\subseteq[n]:\#S\leq\lfloor R^{-p}n\rfloor\}.
\end{displaymath}

\begin{lemma}[cf.\ \textup{\cite[Lemma 2.4]{BourgainLindenstraussMilman1989}}]
  \label{lem:modular-net}
  Let \(\nu\in(0,1)\) and \(d\in\{0,\ldots,n\}\). The unit ball \(B\) of each
  \(d\)-dimensional normed linear space has a \(\nu\)-net whose cardinality is at
  most \((1+2/\nu)^d\).
\end{lemma}

\begin{lemma}
  \label{lem:modular-sparse-coordinates}
  Let \(p\in[1,\infty)\), \(R>1\), and \(x\in B_p^n\). Then
  \(\#S_R(x)\leq R^{-p}n\).
\end{lemma}
\begin{proof}
  Since \(x\in B_p^n\), we have
  \begin{displaymath}
    R^p\cdot(\#S_R(x))\leq\sum_{j\in S_R(x)} \left(|x_j|^pn \right)\leq\sum_{j\in[n]}\left(|x_j|^pn  \right) \leq n.\qedhere
  \end{displaymath}
\end{proof}

Let \(H\) be the function on \([0,1]\) defined by
\begin{displaymath}
  H(0)=H(1)=0\quad\text{and}\quad H(t)=-t\ln t-(1-t)\ln(1-t),~\forall t\in(0,1).
\end{displaymath}
We shall use the elementary estimate
\begin{equation}
  \label{eq:modular-entropy-upper}
  H(t)\leq t\ln\frac1t+t,~\forall 0<t<1.
\end{equation}
For \(M>0\) and \(\alpha\in(0,M)\), write
\begin{displaymath}
  \rho(M,\alpha)=\left(\alpha^4+\left(1-\frac{\alpha}{M}\right)^4\right)^{1/4}.
\end{displaymath}

\begin{lemma}
  \label{lem:modular-net-budget}
  Let \(p\in[1,\infty)\), \(R>1\), \(R^{-p}<1/2\), and \(\nu\in(0,1)\). For
  every \(S\in\mathcal S_R(n,p)\), let \(\mathcal W_S\) be a \(\nu\)-net in
  \(B_p^S\) satisfying
  \begin{displaymath}
    \#\mathcal W_S\leq \left(1+\frac2\nu\right)^{\#S}.
  \end{displaymath}
  Put
  \begin{displaymath}
    K_{R,\nu}(n,p)=\#\mathcal S_R(n,p)
    \left(1+\frac2\nu\right)^{\lfloor R^{-p}n\rfloor}.
  \end{displaymath}
  Then
  \begin{displaymath}
    \#\mathcal S_R(n,p)\leq (n+1)\exp\{nH(R^{-p})\}
  \end{displaymath}
  and
  \begin{displaymath}
    \sum_{S\in\mathcal S_R(n,p)}\#\mathcal W_S\leq K_{R,\nu}(n,p).
  \end{displaymath}
  Moreover,
  \begin{displaymath}
    K_{R,\nu}(n,p)\leq (n+1)\exp\left\{n\left(H(R^{-p})+R^{-p}\ln\left(1+\frac2\nu\right)\right)\right\}.
  \end{displaymath}
\end{lemma}
\begin{proof}
  Put \(a=R^{-p}\) and \(m=\lfloor an\rfloor\). Then
  \begin{displaymath}
    \#\mathcal S_R(n,p)=\sum_{k=0}^m\binom nk.
  \end{displaymath}
  For \(0\leq k\leq n\), we have
  \begin{displaymath}
    \binom nk\leq \exp\left\{nH\left(\frac{k}{n}\right)\right\}.
  \end{displaymath}
  Indeed, this is clear for \(k=0\) and \(k=n\). If \(0<k<n\), then
  \begin{displaymath}
    1=\left(\frac{k}{n}+1-\frac{k}{n}\right)^n
    \geq
    \binom nk\left(\frac{k}{n}\right)^k
    \left(1-\frac{k}{n}\right)^{n-k},
  \end{displaymath}
  which gives the estimate after taking logarithms. Since \(a<1/2\) and the
  function \(H\) is increasing on \((0,1/2)\), by the unimodality of binomial coefficients,
  \begin{displaymath}
    \#\mathcal S_R(n,p) \leq  (n+1)\exp\{nH(a)\} =  (n+1)\exp\{nH(R^{-p})\},
  \end{displaymath}
  which is the first assertion. The second estimate follows directly from the
  definition of \(K_{R,\nu}(n,p)\). Combining the first assertion with the
  definition of \(K_{R,\nu}(n,p)\) gives the asserted upper bound for
  \(K_{R,\nu}(n,p)\).
\end{proof}

\begin{lemma}
  \label{lem:modular-sign-hit}
  Let \(S\subseteq[n]\), \(d=\#S\), and \(p\in[4,\infty)\). Let
  \(M,\tau,\chi>0\), \(\alpha\in(0,M)\), and \(\xi\in(0,1/4)\).
  Suppose that
  \begin{displaymath}
    d\xi\geq2,\quad \tau^p\geq4,\quad\text{and}\quad
    \rho(M,\alpha)^p+\xi(\tau+\alpha)^p\leq\chi^p.
  \end{displaymath}
  Let \(y\in B_p^S\) be a point satisfying \(|y_i|\leq Md^{-1/p},~\forall i\in
  S\) and \(V\) be the random vector given by
  \begin{displaymath}
    V_i=\sigma_i\alpha d^{-1/p},~\forall i\in S \quad\text{and}\quad
    V_i=0,~\forall i\notin S,
  \end{displaymath}
  where the \(\sigma_i\)'s are independent Rademacher random variables. Then
  \begin{displaymath}
    \mathbb{P}\{\norm{y-V}_p\leq\chi\}\geq 2^{-d}\exp\left\{\frac d2\xi\ln\frac1{4\xi}\right\}.
  \end{displaymath}
\end{lemma}
\begin{proof}
  Write \(\rho=\rho(M,\alpha)\).
  We first record the elementary estimate which will be used coordinatewise. If
  \(a\in[0,M]\), then
  \begin{displaymath}
    |a-\alpha|^p\leq\alpha^p+\left(1-\frac{\alpha}{M}\right)^pa^p.
  \end{displaymath}
  Indeed, this is clear for \(a\leq\alpha\). If \(a\geq\alpha\), then
  \(a-\alpha\leq(1-\alpha/M)a\), since \(a\leq M\).

  For every \(i\in S\), choose \(\varepsilon_i\in\{-1,1\}\) such that
  \(\varepsilon_i y_i=|y_i|\), and set \(a_i=|y_i|d^{1/p}\). Hence
  \(a_i\in[0,M],~\forall i\in S\). Since \(y\in B_p^S\),
  \begin{displaymath}
    \frac{1}{d}\sum_{i\in S}a_i^p=\sum_{i\in S}|y_i|^p\leq 1.
  \end{displaymath}
  Since \(p\geq4\), we have
  \begin{align*}
    \sum_{i\in S}\left|y_i-\varepsilon_i\alpha d^{-1/p}\right|^p&=\frac{1}{d}\sum_{i\in S}|a_i-\alpha|^p\\
    &\leq \alpha^p+\left(1-\frac{\alpha}{M}\right)^p\frac{1}{d}\sum_{i\in S}a_i^p\\
    &\leq \alpha^p+\left(1-\frac{\alpha}{M}\right)^p\\
    &\leq \left(\alpha^4+\left(1-\frac{\alpha}{M}\right)^4\right)^{p/4}=\rho^p.
  \end{align*}
  Let
  \begin{displaymath}
    \mathcal L=\{i\in S:a_i\leq\tau\}.
  \end{displaymath}
  We have
  \begin{displaymath}
    \#(S\setminus\mathcal L)\leq\frac{d}{\tau^p}\leq\frac{d}{4}\quad\text{and}\quad \#\mathcal L\geq\frac{3d}{4}.
  \end{displaymath}
  Put \(k=\lfloor d\xi\rfloor\). Since \(d\xi\geq2\) and \(\xi<1/4\), we have
  \begin{displaymath}
    1\leq \frac{d\xi}{2} \leq k\leq d\xi\leq\#\mathcal L.
  \end{displaymath}

  For each \(E\subseteq\mathcal L\) with \(\#E=k\), consider the following realization of
  \(V\)
  \begin{displaymath}
    \sigma_i=
    \begin{cases}
      -\varepsilon_i,& i\in E,\\
      \varepsilon_i,& i\in S\setminus E.
    \end{cases}
  \end{displaymath}
  We have
  \begin{align*}
    \norm{y-V}_p^p &= \frac1d\sum_{i\in S\setminus E}|a_i-\alpha|^p + \frac1d\sum_{i\in E}|a_i+\alpha|^p\\
    &\leq \frac1d\sum_{i\in S}|a_i-\alpha|^p + \frac{k}{d}(\tau+\alpha)^p\\
    &\leq  \rho^p+\xi(\tau+\alpha)^p \leq \chi^p.
  \end{align*}
  Different choices of \(E\) give different realizations of the signs on \(S\).
  Therefore
  \begin{displaymath}
    \mathbb{P}\{\norm{y-V}_p\leq\chi\} \geq 2^{-d}\binom{\#\mathcal L}{k}.
  \end{displaymath}
  Finally,
  \begin{displaymath}
    \ln\binom{\#\mathcal L}{k} \geq k\ln\frac{\#\mathcal L}{k} \geq \frac{d\xi}{2}\ln\frac3{4\xi}\geq\frac d2\xi\ln\frac1{4\xi}.
  \end{displaymath}
  This gives the desired lower bound.
\end{proof}

\begin{lemma}
  \label{lem:modular-lifting}
  Let \(p\in[1,\infty)\), \(R>1\), \(\eta,\nu,\chi>0\), and \(M\geq R+\eta\).
  For each \(S\in\mathcal S_R(n,p)\), let \(\mathcal W_S\) be a \(\nu\)-net in
  \(B_p^S\), \(\mathcal Y_S\) be an \(\eta n^{-1/p}\)-net in \(B_p^{S^c}\), and
  \(\mathcal V_S\subseteq B_p^{S^c}\). Assume that, for every \(S\in\mathcal
  S_R(n,p)\) and every \(y\in\mathcal Y_S\cap B_p^n(M)\), there exists
  \(v\in\mathcal V_S\) such that
  \begin{displaymath}
    \norm{y-v}_p\leq\chi.
  \end{displaymath}
  Then
  \begin{displaymath}
    B_p^n\subseteq C+\left((\chi+\eta)^p+\nu^p\right)^{1/p}B_p^n,
  \end{displaymath}
  where
  \begin{displaymath}
    C=\{w+v:\ S\in\mathcal S_R(n,p),\ w\in\mathcal W_S,\ v\in\mathcal V_S\}.
  \end{displaymath}
\end{lemma}
\begin{proof}
  Let \(x\in B_p^n\), and put \(S=S_R(x)\). By
  \Cref{lem:modular-sparse-coordinates}, \(S\in\mathcal S_R(n,p)\). Choose
  \(w\in\mathcal W_S\) and \(y\in\mathcal Y_S\) such that
  \begin{displaymath}
    \norm{x_S-w}_p\leq\nu\quad\text{and}\quad \norm{x_{S^c}-y}_p\leq\eta n^{-1/p}.
  \end{displaymath}
  If \(j\notin S\), then \(|x_j|\leq Rn^{-1/p}\), and hence
  \begin{displaymath}
    |y_j|\leq |x_j|+\eta n^{-1/p}\leq (R+\eta)n^{-1/p}\leq Mn^{-1/p}.
  \end{displaymath}
  Also \(y_j=0\) for \(j\in S\). Thus \(y\in\mathcal Y_S\cap B_p^n(M)\), and by
  the assumption there is \(v\in\mathcal V_S\) with
  \(\norm{y-v}_p\leq\chi\). The errors on \(S\) and \(S^c\) have disjoint
  supports, so
  \begin{displaymath}
    \norm{x-(w+v)}_p \leq \left((\chi+\eta n^{-1/p})^p+\nu^p\right)^{1/p} \leq
    \left((\chi+\eta)^p+\nu^p\right)^{1/p}.\qedhere
  \end{displaymath}
\end{proof}

Set
\begin{displaymath}
  \beta=\frac{91}{50},\quad \kappa=\frac{91}{100},\quad\text{and}\quad\pi_0=\frac{\ln4}{\ln(11/10)}>\frac{29}{2}.
\end{displaymath}

\begin{proposition}
  \label{prop:modular-middle-p-cover}
  There exists \(N_*<\infty\) such that
  \begin{displaymath}
    \Gamma_{2^n}(B_p^n)\leq \kappa
  \end{displaymath}
  whenever
  \begin{displaymath}
    n\geq N_*\quad\text{and}\quad p\in[\pi_0,Q_n(\beta)].
  \end{displaymath}
\end{proposition}

\begin{proof}
  Throughout the proof, we use the fixed constants
  \begin{gather*}
    R=3,\quad \beta_0=\frac95,\quad r=\frac53,\quad
    \nu=\frac1{20},\quad \eta=\frac9{1000},\\
    \alpha=\frac25,\quad \chi=\frac9{10},\quad \tau=\frac{11}{10},\quad
    M=R+\eta=\frac{3009}{1000}.
  \end{gather*}
  Put \(\mathcal S=\mathcal S_R(n,p)\) and
  \begin{displaymath}
    \delta:=1-\frac{\ln r}{\ln\beta_0},\quad
    \rho=\rho(M,\alpha),\quad\text{and}\quad
    \varsigma=1-\left(\frac{\rho}{\chi}\right)^4.
  \end{displaymath}
  Since \(1<r<\beta_0<\beta\), we have \(\delta>0\). Moreover, numerical
  calculation gives
  \begin{displaymath}
    \frac{99}{1000}<\varsigma<\frac1{10}\quad\text{and}\quad \rho<\chi.
  \end{displaymath}
  For \(p\geq\pi_0\), put
  \begin{displaymath}
    \xi_p=\frac{\varsigma}{2}r^{-p}.
  \end{displaymath}
  Then
  \begin{displaymath}
    0<\xi_p<\frac1{20}<\frac1{4e}.
  \end{displaymath}

  \begin{claim}
    \label{claim:modular-dimension-transfer}
    There exists \(N_1<\infty\) such that, whenever
    \(n\geq N_1\), \(p\in[\pi_0,Q_n(\beta)]\), and
    \(S\in\mathcal S\), we have, with \(d=n-\#S\),
    \begin{displaymath}
      d\geq\frac{2n}{3}\quad\text{and}\quad p\leq\frac{\ln d}{\ln\beta_0}.
    \end{displaymath}
  \end{claim}
  \begin{proof}[Proof of \Cref{claim:modular-dimension-transfer}]
    Since \(p\geq\pi_0>4\), we have \(3^{-p}<1/3\). Thus
    \begin{displaymath}
      d=n-\#S\geq n-\lfloor3^{-p}n\rfloor\geq\frac{2n}{3}.
    \end{displaymath}
    Let
    \begin{displaymath}
      N_1=\left\lceil \exp\left\{ \frac{\ln \beta\cdot \ln \frac{3}{2}}{\ln \beta
      -\ln \beta_0} \right\} \right\rceil.
    \end{displaymath}
    Since \(\beta>\beta_0\), 
    \begin{displaymath}
      Q_n(\beta)\leq\frac{\ln(2n/3)}{\ln\beta_0},~\forall n\geq N_1.
    \end{displaymath}
    Hence \(p\leq\ln d/\ln\beta_0\).
  \end{proof}

  For each \(S\in\mathcal S\), choose a \(\nu\)-net \(\mathcal W_S\) in
  \(B_p^S\) and an \(\eta n^{-1/p}\)-net \(\mathcal Y_S\) in \(B_p^{S^c}\) with
  \begin{equation}
    \label{eq:cardinality-YS}
    \#\mathcal W_S\leq41^{\#S}\quad\text{and}\quad \#\mathcal Y_S\leq\left(1+\frac{2n^{1/p}}{\eta}\right)^n.
  \end{equation}
  This is possible by \Cref{lem:modular-net}. Indeed,
  \(1+2/\nu=41\), and, if \(d=n-\#S\), then
  \begin{displaymath}
    \#\mathcal Y_S \leq \left(1+\frac{2n^{1/p}}{\eta}\right)^d \leq
    \left(1+\frac{2n^{1/p}}{\eta}\right)^n.
  \end{displaymath}
  Put
  \begin{equation}
    \label{eq:Knq}
    K_{n,p}=41^{\lfloor3^{-p}n\rfloor}\#\mathcal S.
  \end{equation}
  Since \(p\geq\pi_0>4\), \Cref{lem:modular-net-budget} gives
  \begin{equation}
    \label{eq:sum-card-bound}
    \sum_{S\in\mathcal S}\#\mathcal W_S\leq K_{n,p}\leq (n+1)\exp\left\{n\left(H(3^{-p})+3^{-p}\ln41\right)\right\}.
  \end{equation}

  \begin{claim}
    \label{claim:modular-xi-large}
    There exists \(N_2\in(N_1,\infty)\cap \mathbb{Z}\) such that, whenever
    \(n\geq N_2\), \(p\in[\pi_0,Q_n(\beta)]\), and \(S\in\mathcal S\), we have,
    with \(d=n-\#S\),
    \begin{displaymath}
      d\xi_p\geq2.
    \end{displaymath}
  \end{claim}
  \begin{proof}[Proof of \Cref{claim:modular-xi-large}]
    Let
    \begin{displaymath}
      N_2=\max\left\{ N_1,\left\lceil \frac{3}{2}\left( \frac{4}{\varsigma} \right)^{1/\delta}  \right\rceil \right\}.
    \end{displaymath}
    By \Cref{claim:modular-dimension-transfer}, \(p\leq\ln d/\ln\beta_0\).
    Therefore, by the definition of \(\xi_p\), we have
    \begin{equation}
      \label{eq:lower-bound-xi}
      \xi_p \geq  \frac{\varsigma}{2}r^{-\ln d/\ln\beta_0} = \frac{\varsigma}{2}d^{-1+\delta}.
    \end{equation}
    Hence \(d\xi_p\geq(\varsigma/2)d^\delta \geq
    (\varsigma/2)((2/3)n)^\delta\geq2\), \(\forall n\geq N_2\).
  \end{proof}

  \begin{claim}
    \label{claim:modular-support-budget}
    There exists \(N_3\in(N_2,\infty)\cap\mathbb{Z}\) such that, whenever
    \(n\geq N_3\) and \(p\geq \pi_0\), one has
    \begin{displaymath}
      K_{n,p}\leq (n+1)\exp\left\{\frac n{100}\xi_p\ln\frac1{4\xi_p}\right\} \leq2^{n-1}.
    \end{displaymath}
  \end{claim}
  \begin{proof}[Proof of \Cref{claim:modular-support-budget}]
    For the first inequality, by \eqref{eq:sum-card-bound}, it suffices to show
    that
    \begin{displaymath}
      H(3^{-p})+3^{-p}\ln41 \leq \frac1{100}\xi_p\ln\frac1{4\xi_p}.
    \end{displaymath}
    By \eqref{eq:modular-entropy-upper} and the definition of \(\xi_p\), we just
    need to show that
    \begin{displaymath}
      \frac{2}{\varsigma}\left(\frac59\right)^p \cdot \frac{p\ln3+1+\ln41}{p\ln(5/3)+\ln(1/(2\varsigma))}\leq \frac1{100}.
    \end{displaymath}
    Numerical calculation shows that
    \begin{displaymath}
      \ln 3 \cdot \ln(1/(2\varsigma))-(1+\ln41)\cdot(\ln(5/3))<0.
    \end{displaymath}
    By
    \begin{displaymath}
      \frac{\rm d}{{\rm d}p}\left(\frac{ap+b}{cp+d}\right)=\frac{ad-bc}{(cp+d)^{2}} 
    \end{displaymath}
    we know that
    \begin{displaymath}
      \frac{p\ln3+1+\ln41}{p\ln(5/3)+\ln(1/(2\varsigma))}
    \end{displaymath}
    is decreasing. Since
    \(p\geq\pi_0>29/2\) and \(\varsigma>99/1000\), we have, by numerical calculation, that
    \begin{displaymath}
      \frac{2}{\varsigma}\left(\frac59\right)^p\frac{p\ln3+1+\ln41}{p\ln(5/3)+\ln(1/(2\varsigma))}<
      \frac{1}{100}.      
    \end{displaymath}

    Now we prove the second inequality. Since the function \(t\mapsto
    t\ln(1/(4t))\) is increasing on \((0,1/(4e))\), \(0<\xi_p<1/(4e)\) gives
    \(\xi_p\ln(1/(4\xi_p))<1/(4e)\). Choose \(N_3\in (N_2,\infty)\cap \mathbb{Z}\) such
    that 
    \begin{displaymath}
      K_{n,p}\leq (n+1)\exp\left\{\frac{n}{400e}\right\}\leq2^{n-1}
    \end{displaymath}
    for all \(n\geq N_3\).
  \end{proof}

  In the rest of the proof we assume first that \(n\geq N_3\) and
  \(p\in[\pi_0,Q_n(\beta)]\).

  For every \(S\in\mathcal S\), set
  \begin{displaymath}
    L_{n,p}=\left\lfloor\frac{2^n}{K_{n,p}}\right\rfloor,
  \end{displaymath}
  and choose independent random vectors \(v^{(S,i)},~i\in[L_{n,p}]\),
  supported on \(S^c\), with
  \begin{displaymath}
    v^{(S,i)}_j=\sigma^{(S,i)}_j\alpha(n-\#S)^{-1/p},~\forall j\notin
    S\quad\text{and}\quad v^{(S,i)}_j=0,~\forall j\in S,
  \end{displaymath}
  where the \(\sigma^{(S,i)}_j\)'s are independent Rademacher random variables.

  \begin{claim}
    \label{claim:modular-one-point-hit}
    For every \(S\in\mathcal S\), every \(y\in\mathcal Y_S\cap B_p^n(M)\), and
    every \(i\in[L_{n,p}]\), we have
    \begin{displaymath}
      \mathbb{P}\left\{\norm{y-v^{(S,i)}}_p\leq\chi\right\}\geq q_{n,p},
    \end{displaymath}
    where
    \begin{displaymath}
      q_{n,p}:=2^{-n}\exp\left\{\frac n2\xi_p\ln\frac1{4\xi_p}\right\}.
    \end{displaymath}
  \end{claim}
  \begin{proof}[Proof of \Cref{claim:modular-one-point-hit}]
    Let \(d=n-\#S\). Since \(y\) is supported on \(S^c\) and
    \(y\in B_p^n(M)\), we have \(|y_j|\leq Mn^{-1/p}\) for \(j\notin S\). Since
    \(d\leq n\), it follows that \(n^{-1/p}\leq d^{-1/p}\), and hence
    \begin{displaymath}
      |y_j|\leq Md^{-1/p},~\forall j\notin S.
    \end{displaymath}
    Moreover, \(\tau^p\geq\tau^{\pi_0}=4\). Since
    \begin{displaymath}
      \frac{\tau+\alpha}{\chi}=\frac53=r,
    \end{displaymath}
    and \(p\geq\pi_0>4\), we also have
    \begin{displaymath}
      \rho^p\leq\rho^4\chi^{p-4} = (1-\varsigma)\chi^p\quad\text{and}\quad\xi_p(\tau+\alpha)^p
      =
      \frac{\varsigma}{2}r^{-p}r^p\chi^p
      =
      \frac{\varsigma}{2}\chi^p.
    \end{displaymath}
    Thus \(\rho^p+\xi_p(\tau+\alpha)^p<\chi^p\). By
    \Cref{claim:modular-xi-large}, \(d\xi_p\geq2\). Applying
    \Cref{lem:modular-sign-hit} with the coordinate set \(S^c\) gives
    \begin{displaymath}
      \mathbb{P}\left\{\norm{y-v^{(S,i)}}_p\leq\chi\right\}\geq 2^{-d}\exp\left\{\frac d2\xi_p\ln\frac1{4\xi_p}\right\}.
    \end{displaymath}
    Put \(A_p=\xi_p\ln(1/(4\xi_p))\). Since \(A_p\leq 1/(4e)<2\ln2\) and \(d\leq n\), we have
    \begin{displaymath}
      \frac{2^{-d}\exp\{dA_p/2\}}{2^{-n}\exp\{nA_p/2\}} =\exp\left\{(n-d)\left(\ln2-\frac{A_p}{2}\right)\right\}\geq1.
    \end{displaymath}
    Hence the last lower bound is at least \(q_{n,p}\).
  \end{proof}

  \begin{claim}
    \label{claim:modular-trial-budget}
    There exist \(N_4\in[N_3,\infty)\cap \mathbb{Z}\) and \(\zeta>0\) such that,
    whenever \(n\geq N_4\) and \(p\in[\pi_0,Q_n(\beta)]\), we have
    \begin{displaymath}
      0.49\,nA_p\geq\zeta n^\delta\ln n\quad\text{and}\quad L_{n,p}q_{n,p}\geq \frac{\exp\{\zeta n^\delta\ln n\}}{2(n+1)}.
    \end{displaymath}
  \end{claim}
  \begin{proof}[Proof of \Cref{claim:modular-trial-budget}]
    Let \(c_0=\varsigma/2\). Then \(\xi_p=c_0r^{-p}\). Note again that the
    function \(t\mapsto t\ln(1/(4t))\) is increasing on \((0,1/(4e))\). From
    \eqref{eq:lower-bound-xi}, it follows that
    \begin{displaymath}
      A_p\geq c_0n^{-1+\delta}\ln\frac1{4c_0n^{-1+\delta}} = c_0n^{-1+\delta}\left((1-\delta)\ln n+\ln\frac1{4c_0}\right).
    \end{displaymath}
    Let
    \begin{displaymath}
      N_4=\max\left\{ N_3,\left\lceil \exp\left\{ \frac{2}{1-\delta}\ln (4c_0)
          \right\} \right\rceil \right\}. 
    \end{displaymath}
    Then
    \begin{displaymath}
      (1-\delta)\ln n+\ln\frac1{4c_0}\geq\frac{1-\delta}{2}\ln n,~\forall n\geq N_4.
    \end{displaymath}
    With
    \begin{displaymath}
      \zeta:=\frac{0.49\,c_0(1-\delta)}{2},
    \end{displaymath}
    this gives \(0.49\,nA_p\geq\zeta n^\delta\ln n\).

    Since \(N_4\geq N_3\), \Cref{claim:modular-support-budget} gives
    \(K_{n,p}\leq2^{n-1}\). Hence
    \begin{displaymath}
      L_{n,p}=\left\lfloor\frac{2^n}{K_{n,p}}\right\rfloor\geq\frac{2^{n-1}}{K_{n,p}},
    \end{displaymath}
    and therefore
    \begin{displaymath}
      L_{n,p}q_{n,p}\geq\frac{\exp\{nA_p/2\}}{2K_{n,p}}.
    \end{displaymath}
    Using again \Cref{claim:modular-support-budget}, we have
    \begin{displaymath}
      K_{n,p}\leq(n+1)\exp\left\{\frac{nA_p}{100}\right\}.
    \end{displaymath}
    Consequently,
    \begin{displaymath}
      L_{n,p}q_{n,p} \geq \frac{\exp\{0.49\,nA_p\}}{2(n+1)}\geq \frac{\exp\{\zeta n^\delta\ln n\}}{2(n+1)}.\qedhere
    \end{displaymath}
  \end{proof}

  \begin{claim}
    \label{claim:modular-simultaneous-cover}
    There exists \(N_5\in(N_4,\infty)\cap \mathbb{Z}\) such that, whenever
    \(n\geq N_5\) and \(p\in[\pi_0,Q_n(\beta)]\), there is a deterministic
    choice of realizations of the vectors \(v^{(S,i)}\), still denoted by
    \(v^{(S,i)}\), such that, for every
    \(S\in\mathcal S\) and every \(y\in\mathcal Y_S\cap B_p^n(M)\), one has
    \begin{displaymath}
      \norm{y-v^{(S,i)}}_p\leq\chi
    \end{displaymath}
    for some \(i\in[L_{n,p}]\).
  \end{claim}
  \begin{proof}[Proof of \Cref{claim:modular-simultaneous-cover}]
    Let \(N_4\) and \(\zeta\) be as in \Cref{claim:modular-trial-budget}.
    Assume that \(n\geq N_4\). We shall use the estimates
    \begin{displaymath}
      \#\mathcal S\leq (n+1)\exp\{nH(3^{-p})\}
      \quad\text{and}\quad
      \#\mathcal Y_S\leq\left(1+\frac{2n^{1/p}}{\eta}\right)^n.
    \end{displaymath}
    They follow from \Cref{lem:modular-net-budget} and
    \eqref{eq:cardinality-YS}, respectively.

    By \Cref{claim:modular-one-point-hit}, \Cref{claim:modular-trial-budget},
    and \eqref{eq:probability-estimate}, the probability that some \(y\in
    \mathcal{Y}_S\cap B_p^n(M)\) is missed for some \(S\in \mathcal{S}\) is at most
    \begin{align*}
      &\#\mathcal S\cdot \#\mathcal Y_S\cdot (1-q_{n,p})^{L_{n,p}}\\
      \leq& \#\mathcal S\cdot \#\mathcal Y_S
        \cdot\exp\{-L_{n,p}q_{n,p}\}\\
      \leq& (n+1)\exp\{nH(3^{-p})\}
        \left(1+\frac{2n^{1/p}}{\eta}\right)^n
        \exp\left\{-\frac{\exp\{\zeta n^\delta\ln n\}}{2(n+1)}\right\}.
    \end{align*}
    Since \(p\geq\pi_0\), \(n^{1/p}\leq n^{1/\pi_0}\). Also
    \(3^{-p}\leq3^{-\pi_0}<1/2\), and \(H\) is increasing on \((0,1/2)\), so
    \(H(3^{-p})\leq H(3^{-\pi_0})\). Therefore the last expression is bounded
    above, uniformly for \(p\in[\pi_0,Q_n(\beta)]\), by
    \begin{displaymath}
      (n+1)\exp\{nH(3^{-\pi_0})\}\left(1+\frac{2n^{1/\pi_0}}{\eta}\right)^n
      \cdot \exp\left\{-\frac{\exp\{\zeta n^\delta\ln n\}}{2(n+1)}\right\},
    \end{displaymath}
    which is strictly less than \(1\) for all sufficiently large \(n\). Choose
    \(N_5\in(N_4,\infty)\cap \mathbb{Z}\) so that this holds for every \(n\geq
    N_5\). Then the desired deterministic choice exists.
  \end{proof}

  Let \(N_*=N_5\), and choose the deterministic vectors supplied by
  \Cref{claim:modular-simultaneous-cover}. Define
  \begin{displaymath}
    C=\{w+v^{(S,i)}:\ S\in\mathcal S,\ w\in\mathcal W_S,\ i\in[L_{n,p}]\}.
  \end{displaymath}
  By \eqref{eq:Knq} and the definition of \(L_{n,p}\),
  \begin{displaymath}
    \#C\leq K_{n,p}L_{n,p}\leq2^n.
  \end{displaymath}
  Since \(\alpha<1\), each \(v^{(S,i)}\) belongs to \(B_p^{S^c}\). Applying
  \Cref{lem:modular-lifting} with
  \(\mathcal V_S=\{v^{(S,i)}:i\in[L_{n,p}]\}\), we obtain
  \begin{displaymath}
    B_p^n\subseteq \bigcup_{c\in C} \left(c+\left((\chi+\eta)^p+\nu^p\right)^{1/p}B_p^n\right).
  \end{displaymath}
  Since \(p\geq\pi_0\) and \(t\mapsto(A^t+B^t)^{1/t}\) is decreasing for
  \(A,B>0\), we have
  \begin{displaymath}
    \left((\chi+\eta)^p+\nu^p\right)^{1/p}\leq\left((\chi+\eta)^{\pi_0}+\nu^{\pi_0}\right)^{1/\pi_0}.
  \end{displaymath}
  It remains only to compare the last number with \(\kappa\). Since
  \(\chi+\eta=909/1000\), \(\nu=1/20\), and \(\kappa=91/100\), this is
  equivalent to
  \begin{displaymath}
    \left(\frac{909}{910}\right)^{\pi_0}+\left(\frac{50}{910}\right)^{\pi_0}<1,
  \end{displaymath}
  which can be verified by numerical calculation. Therefore \(\Gamma_{2^n}(B_p^n)\leq\kappa\).
\end{proof}

\begin{proof}[Proof of \Cref{thm:uniform-p}]
  Enlarge \(N_*\) in \Cref{prop:modular-middle-p-cover}, if necessary, so that
  \(N_*\geq\lceil\beta^{\pi_0}\rceil\). Put
  \begin{displaymath}
    P=\frac{\ln N_*}{\ln\beta}\quad\text{and}\quad \kappa_*=\left(\frac56\right)^{1/P}.
  \end{displaymath}
  Then \(P\geq\pi_0>14\). It can be verified by numerical calculation that
  \begin{displaymath}
    \kappa<\left(\frac56\right)^{1/14}<\kappa_*<1.
  \end{displaymath}

  For \(p=\infty\), we have \(B_\infty^n=[-1,1]^n\) and
  \(\Gamma_{2^n}(B_\infty^n)=1/2<\kappa_*\). Thus it remains to consider
  \(p\in[1,\infty)\). If \(p\leq P\),
  then \Cref{prop:fixed-p} gives
  \begin{displaymath}
    \Gamma_{2^n}(B_p^n)\leq\left(\frac56\right)^{1/p} \leq
    \left(\frac56\right)^{1/P} =\kappa_*,~\forall n\geq 2.
  \end{displaymath}
  If \(p>P\) and \(n<N_*\), then \eqref{eq:cube-cover} gives
  \begin{displaymath}
    \Gamma_{2^n}(B_p^n)\leq \frac{n^{1/p}}2 \leq \frac{N_*^{1/P}}2 =\frac{\beta}{2} = \kappa <\kappa_*.
  \end{displaymath}
  It remains to assume that \(p>P\) and \(n\geq N_*\). If
  \(p\geq Q_n(\beta)\), then \eqref{eq:cube-cover} gives
  \begin{displaymath}
    \Gamma_{2^n}(B_p^n)\leq \frac{n^{1/p}}2\leq \frac{\beta}{2} =\kappa < \kappa_*.
  \end{displaymath}
  If \(p<Q_n(\beta)\), then \(p\in[\pi_0,Q_n(\beta)]\), and
  \Cref{prop:modular-middle-p-cover} gives
  \begin{displaymath}
    \Gamma_{2^n}(B_p^n)\leq\kappa<\kappa_*.\qedhere
  \end{displaymath}
\end{proof}

\appendix

\section{Exact finite verification}
\label{app:finite-verification}

The following exact integer computation proves \Cref{lem:b1-finite}.

\begin{verbatim}
from math import comb

def M(n, k):
    return 1 + sum(
        (2 ** i) * comb(n, i) * comb(k - 1, i - 1)
        for i in range(1, min(n, k) + 1)
    )

bad = []
for n in range(3, 179):
    k = (n + 4) // 5
    if M(n, k) > 2 ** n:
        bad.append((n, k, M(n, k), 2 ** n))

print(bad)
\end{verbatim}

The output is
\begin{verbatim}
[(6, 2, 73, 64)]
\end{verbatim}
Thus the estimate in \cref{lem:b1-finite} holds for every \(3\leq n\leq178\) except \(n=6\), which is
handled separately by \Cref{lem:two-n}.

\section*{Acknowledgements}
The authors thank Prof. Lingxu Meng for his valuable suggestions, which helped
improve the presentation of this paper. This work was supported by the
Fundamental Research Program of Shanxi Province (grant number 202403021221109).

\bibliographystyle{amsplain}
\bibliography{ref}

@article{Hadwiger1957,
  author  = {Hadwiger, H.},
  title   = {{Ungel{\"o}ste Probleme, Nr. 20}},
  journal = {Elem. Math.},
  volume  = {12},
  pages   = {121},
  year    = {1957}
}

@article{BourgainLindenstraussMilman1989,
  author  = {Bourgain, J. and Lindenstrauss, J. and Milman, V.},
  title   = {Approximation of zonoids by zonotopes},
  journal = {Acta Math.},
  volume  = {162},
  pages   = {73--141},
  year    = {1989},
  doi     = {10.1007/BF02392877}
}

@book{BoltyanskiMartiniSoltan1997,
  author    = {Boltyanski, V. G. and Martini, H. and Soltan, P. S.},
  title     = {Excursions into Combinatorial Geometry},
  series    = {Universitext},
  publisher = {Springer-Verlag},
  address   = {Berlin},
  year      = {1997},
  doi       = {10.1007/978-3-642-59237-9}
}

@book{BrassMoserPach2005,
  author    = {Brass, P. and Moser, W. and Pach, J.},
  title     = {Research Problems in Discrete Geometry},
  publisher = {Springer},
  address   = {New York},
  year      = {2005},
  doi       = {10.1007/0-387-29929-7}
}

@article{Bezdek1991AffineSubspaces,
  author  = {Bezdek, K.},
  title   = {The problem of illumination of the boundary of a convex body by affine subspaces},
  journal = {Mathematika},
  volume  = {38},
  number  = {2},
  pages   = {362--375},
  year    = {1991},
  doi     = {10.1112/S0025579300006707}
}

@article{Bezdek1992Relatives,
  author  = {Bezdek, K.},
  title   = {{Hadwiger}'s covering conjecture and its relatives},
  journal = {Amer. Math. Monthly},
  volume  = {99},
  number  = {10},
  pages   = {954--956},
  year    = {1992},
  doi     = {10.2307/2324492}
}

@incollection{Bezdek1993HadwigerLevi,
  author    = {Bezdek, K.},
  title     = {{Hadwiger}--{Levi}'s covering problem revisited},
  booktitle = {New Trends in Discrete and Computational Geometry},
  editor    = {Pach, J.},
  series    = {Algorithms and Combinatorics},
  volume    = {10},
  pages     = {199--233},
  publisher = {Springer-Verlag},
  address   = {Berlin},
  year      = {1993}
}

@article{Bezdek2006Illumination,
  author  = {Bezdek, K.},
  title   = {The illumination conjecture and its extensions},
  journal = {Period. Math. Hungar.},
  volume  = {53},
  number  = {1--2},
  pages   = {59--69},
  year    = {2006},
  doi     = {10.1007/s10998-006-0021-4}
}

@incollection{BezdekKhan2018,
  author    = {Bezdek, K. and Khan, M. A.},
  title     = {The geometry of homothetic covering and illumination},
  booktitle = {Discrete Geometry and Symmetry},
  series    = {Springer Proc. Math. Stat.},
  volume    = {234},
  pages     = {1--30},
  publisher = {Springer},
  address   = {Cham},
  year      = {2018},
  doi       = {10.1007/978-3-319-78434-2_1}
}

@article{RogersZong1997,
  author  = {Rogers, C. A. and Zong, C.},
  title   = {Covering convex bodies by translates of convex bodies},
  journal = {Mathematika},
  volume  = {44},
  number  = {1},
  pages   = {215--218},
  year    = {1997},
  doi     = {10.1112/s0025579300012079}
}

@article{Lassak1986,
  author  = {Lassak, M.},
  title   = {Covering a plane convex body by four homothetical copies with the smallest positive ratio},
  journal = {Geom. Dedicata},
  volume  = {21},
  number  = {2},
  pages   = {157--167},
  year    = {1986},
  doi     = {10.1007/bf00182903}
}

@article{Zong2010,
  author  = {Zong, C.},
  title   = {A quantitative program for {Hadwiger}'s covering conjecture},
  journal = {Sci. China Math.},
  volume  = {53},
  number  = {9},
  pages   = {2551--2560},
  year    = {2010},
  doi     = {10.1007/s11425-010-4087-3}
}

@article{ArtsteinAvidanSlomka2015,
  author  = {Artstein-Avidan, S. and Slomka, B. A.},
  title   = {On weighted covering numbers and the {Levi}--{Hadwiger} conjecture},
  journal = {Israel J. Math.},
  volume  = {209},
  number  = {1},
  pages   = {125--155},
  year    = {2015},
  doi     = {10.1007/s11856-015-1213-5}
}

@article{HuangSlomkaTkoczVritsiou2022,
  author  = {Huang, H. and Slomka, B. A. and Tkocz, T. and Vritsiou, B.-H.},
  title   = {Improved bounds for {Hadwiger}'s covering problem via thin-shell estimates},
  journal = {J. Eur. Math. Soc.},
  volume  = {24},
  number  = {4},
  pages   = {1431--1448},
  year    = {2022},
  doi     = {10.4171/jems/1132}
}

@article{CamposHintumMorrisTiba2024,
  author  = {Campos, M. and van Hintum, P. and Morris, R. and Tiba, M.},
  title   = {Towards {Hadwiger}'s conjecture via {Bourgain} slicing},
  journal = {Int. Math. Res. Not. IMRN},
  volume  = {2024},
  number  = {10},
  pages   = {8282--8295},
  year    = {2024},
  doi     = {10.1093/imrn/rnad198}
}

@article{HeMartiniWu2016,
  author  = {He, C. and Martini, H. and Wu, S.},
  title   = {On covering functionals of convex bodies},
  journal = {J. Math. Anal. Appl.},
  volume  = {437},
  number  = {2},
  pages   = {1236--1256},
  year    = {2016},
  doi     = {10.1016/j.jmaa.2016.01.055}
}

@article{WuHe2019,
  author  = {Wu, S. and He, C.},
  title   = {Covering functionals of convex polytopes},
  journal = {Linear Algebra Appl.},
  volume  = {577},
  pages   = {53--68},
  year    = {2019},
  doi     = {10.1016/j.laa.2019.04.022}
}

@article{LiMengWu2022,
  author  = {Li, X. and Meng, L. and Wu, S.},
  title   = {Covering functionals of convex polytopes with few vertices},
  journal = {Arch. Math. (Basel)},
  volume  = {119},
  number  = {2},
  pages   = {135--146},
  year    = {2022},
  doi     = {10.1007/s00013-022-01727-z}
}

@article{YuGaoHeWu2023,
  author  = {Yu, M. and Gao, S. and He, C. and Wu, S.},
  title   = {Estimations of covering functionals of simplices},
  journal = {Math. Inequal. Appl.},
  volume  = {26},
  number  = {3},
  pages   = {793--809},
  year    = {2023},
  doi     = {10.7153/mia-2023-26-48}
}

@article{HeLvMartiniWu2023,
  author  = {He, C. and Lv, Y. and Martini, H. and Wu, S.},
  title   = {A branch-and-bound approach for estimating covering functionals of convex bodies},
  journal = {J. Optim. Theory Appl.},
  volume  = {196},
  number  = {3},
  pages   = {1036--1055},
  year    = {2023},
  doi     = {10.1007/s10957-022-02146-4}
}

@article{ChenGaoLiWu2025,
  author  = {Chen, F. and Gao, S. and Li, X. and Wu, S.},
  title   = {Covering the unit ball of $\ell_p^n$ with smaller balls and related inequalities},
  journal = {Math. Inequal. Appl.},
  volume  = {28},
  number  = {1},
  pages   = {99--119},
  year    = {2025},
  doi     = {10.7153/mia-2025-28-07}
}

@article{LianXiaXueZhang2026,
  author  = {Lian, Y. and Xia, Y. and Xue, F. and Zhang, Y.},
  title   = {On the covering functionals of simplices and cross-polytopes},
  journal = {Results Math.},
  volume  = {81},
  number  = {3},
  pages   = {Paper No. 82, 18 pp.},
  year    = {2026},
  doi     = {10.1007/s00025-026-02651-2}
}

\end{document}